\documentclass[12pt]{article}
\usepackage{amsmath,amsfonts}
\usepackage{algorithm}
\usepackage{algpseudocode}
\usepackage{verbatim}
\usepackage{latexsym}
\usepackage{graphicx}
\usepackage{mathrsfs} 
\usepackage{amsmath}
\usepackage{stmaryrd}
\usepackage{xcolor}
\usepackage{bm}
\usepackage{cite}
\usepackage{epsfig,psfrag,amsmath,amssymb,url,bm,tikz}
\usepackage{amsfonts,amssymb,amsmath,oldgerm}
\usepackage{epsfig}
\numberwithin{equation}{section}
\usepackage[thinlines]{easybmat}
\usepackage{graphicx}
\usepackage{paralist}
\usepackage[inline]{enumitem}

\makeatletter
\@addtoreset{equation}{section}

\begin{document}
\newcommand{\E}{\mathbb{E}}
\newcommand{\PP}{\mathbb{P}}
\newcommand{\RR}{\mathbb{R}}

\newtheorem{theorem}{Theorem}[section] 
\newtheorem{remark}[theorem]{Remark}
\newtheorem{lemma}[theorem]{Lemma} 
\newtheorem{coro}[theorem]{Corollary}
\newtheorem{defn}[theorem]{Definition}
\newtheorem{assp}[theorem]{Assumption}
\newtheorem{expl}[theorem]{Example}
\newtheorem{prop}[theorem]{Proposition}
\newtheorem{rmk}[theorem]{Remark}
\newtheorem{assumption}[theorem]{Assumption}
\newcommand\tq{{\scriptstyle{3\over 4 }\scriptstyle}}
\newcommand\qua{{\scriptstyle{1\over 4 }\scriptstyle}}
\newcommand\hf{{\textstyle{1\over 2 }\displaystyle}}
\newcommand\hhf{{\scriptstyle{1\over 2 }\scriptstyle}}
\newcommand{\LA}[1]{\ref{lemma:#1}}

\newcommand{\proof}{\noindent {\it Proof}. } 
\newcommand{\eproof}{\hfill $\Box$} 

\def\tl{\tilde}
\def\x{{\bm x}}
\def\y{{\bm y}}
\def\trace{\hbox{\rm trace}}
\def\diag{\hbox{\rm diag}}
\def\for{\quad\hbox{for }}
\def\refer{\hangindent=0.3in\hangafter=1}
\newcommand{\vvec}[1]{{\mathbf{#1}}}
\newcommand\wD{\widehat{\D}}
\def\FGhat{{\widehat{F}_{\Gamma}}}
\def\FGtilde{{\widetilde{F}_{\Gamma}}}
\def\FG{{F_{\Gamma}}}
\def\FGhatc{{\widehat{F}_{c,\widehat{\Gamma}}}}
\def\FGtildec{{\widetilde{F}_{c,\widehat{\Gamma}}}}
\def\FGc{{F_{c,\widehat{\Gamma}}}}
\renewcommand{\div}{\mathop{\sf div  }}
\def\MQ{{\mathcal Q}}
\def\MU{{\mathcal U}}
\def\A{{\cal A}}
\def\B{{\cal B}}
\def\Bhat{{\hat{\cal B}}}
\def\C{{\cal C}}
\def\D{{\cal D}}
\def\E{{\cal E}}
\def\Real{{\bf R}}
\def\T{{\cal T}}
\def\G{{\cal G}}
\def\q{{\bm q}}
\def\bfw{{\bm w}}
\def\bfW{{\bf W}}
\def\F{{\cal F}}
\def\H{{\cal H}}
\def\L{{\cal L}}
\def\bfL{{\bf L}}
\def\U{{\cal U}}
\def\O{{\cal O}}
\def\uhat{{\widehat{{\bf u}}}}
\def\qhat{{\widehat{\bf q}}}
\def\Hhat{{\hat{\cal H}}}
\def\Htilde{{\tilde{\cal H}}}
\def\I{{\cal I}}
\def\Ktilde{{\widetilde{K}}}
\def\Khat{{\widehat{K}}}
\def\Kbar{{\overline{K}}}
\def\barlambda{{\overline{{\bm\lambda}}}}
\def\Kbreve{{\breve{K}}}
\def\M{{\cal M}}
\def\Mhat{{\widehat{M}}}
\def\N{{\cal N}}
\def\Khat{{\widehat{K}}}
\def\PDr{{P_{\delta}}}
\def\PBD{{P_{_D}}}
\def\Q{{\cal Q}}
\def\bfQ{{\bf {\cal Q}}}
\def\RBM{{\rm RBM}}
\def\V{{\cal V}}
\def\W{{\cal W}}
\def\What{{\widehat{W}}}
\def\PDr{{P_{\delta}}}
\def\Stilde{{\widetilde{S}}}
\def\Shat{{\widehat{S}}}
\def\Stildeeps{{\widetilde{S}}_{\large\vardelta}}
\def\Seps{S_{\large\vardelta}}
\def\Sepsdelta{S_{{\vardelta},\delta}}
\def\Sdelta{S_{\delta}}
\def\Stildeepsdelta{{\widetilde{S}}_{{\large\vardelta},\delta}}
\def\Wtilde{{\widetilde{W}}}
\def\Rtilde{{\widetilde{R}}}
\def\Rbar{{\widebar{R}}}
\def\ker{{\rm \bf ker \, }}
\def\range{{\rm \bf range \, }}
\def\diver{{\rm div \, }}
\def\rot{{\rm rot \, }}
\def\>{\raisebox{-1ex}{$\; \stackrel{\textstyle >}{\sim } \; $}}
\def\<{\raisebox{-1ex}{$ \; \stackrel{\textstyle <} {\sim } \; $}}
\def\endproof{\begin{flushright} $\Box $ \end{flushright}}
\def\uG{{{\bm\mu}_\Gamma}}
\def\uE{{{\bm\mu}_{{\mathcal O},\Gamma}}}
 \def\pE{{p_{{\mathcal A},\Gamma}}}
\def\vE{{{\bf v}_{{\mathcal A},\Gamma}}}
\def\wE{{{\bm w}_{{\mathcal O},\Gamma}}}
\def\uH{{{\bf u}_{{\mathcal H},\Gamma}}}
\def\wG{{\chi_{{\bm w}_\Gamma}}}
\def\wGT{{\chi^T_{{\bm w}_\Gamma}}}
\def\Imap{{I_H^{\Omega^i}}}
\def\Imapj{{I_H^{\Omega^j}}}
\def\Ibmap{{I_H^{\partial\Omega^i}}}
\def\Ibmapj{{I_H^{\partial\Omega^j}}}
\def\Ibmapk{{I_H^{\partial\Omega^k}}}
\def\Ibmapl{{I_H^{\partial\Omega^l}}}
\def\Ktilde{{\widetilde{K}}}
\def\Btilde{{\widetilde{B}}}
\def\Otilde{{\widetilde{O}}}
\def\phitilde{{\widetilde{\pmb{\varphi}}}}
\def\Ntilde{{\widetilde{N}}}
\def\Ztilde{{\widetilde{Z}}}
\def\Atilde{{\widetilde{A}}}
\def\Khat{{\widehat{K}}}
\def\Kbar{{\overline{K}}}
\def\Kbreve{{\breve{K}}}
\def\M{{\cal M}}
\def\S{\mathscr{S}}
\def\Mhat{{\widehat{M}}}
\def\N{{\mathcal N}}
\def\Khat{{\widehat{K}}}
\def\PDr{{P_{\delta}}}
\def\PBD{{P_{_D}}}
\def\Q{{\cal Q}}
\def\RBM{{\rm RBM}}
\def\V{{\cal V}}
\def\W{{\cal W}}
\def\What{{\widehat{W}}}
\def\wbar{{\overline{w}}}
\def\PDr{{P_{\delta}}}
\def\Stilde{{\widetilde{S}}}
\def\STG{{\Stilde_\Gamma}}
\def\KTG{{\widetilde{K}_\Gamma}}
\def\BTG{{\widetilde{B}_\Gamma}}
\def\NTG{{\widetilde{N}_\Gamma}}
\def\ZTG{{\widetilde{Z}_\Gamma}}
\def\Utilde{{\widetilde{U}}}
\def\Stildeeps{{\widetilde{S}}_{\large\vardelta}}
\def\Seps{S_{\large\vardelta}}
\def\Sepsdelta{S_{{\vardelta},\delta}}
\def\Sdelta{S_{\delta}}
\def\Stildeepsdelta{{\widetilde{S}}_{{\large\vardelta},\delta}}
\def\Wtilde{{\widetilde{W}}}
\def\Wbar{{\overline{W}}}
\def\ker{{\rm \bf ker \, }}
\def\range{{\rm \bf range \, }}
\def\diver{{\rm div \, }}
\def\rot{{\rm rot \, }}
\def\>{\raisebox{-1ex}{$\; \stackrel{\textstyle >}{\sim } \; $}}
\def\<{\raisebox{-1ex}{$ \; \stackrel{\textstyle <} {\sim } \; $}}
\def\u{{{\bf u}}}
\def\v{{{\bf v}}}
\def\q{{{\bm q}}}
\def\bfm{{{\bm m}}}
\def\SGS{{B_\Gamma}}
\def\SGSS{{Z_\Gamma}}
\def\MK{{M_{KL}}}
\def\MKi{{M^{(i)}_{KL}}}
\def\gtilde{{\widetilde{g}}}
\def\ctilde{{\tilde{c}}}
\def\ytilde{{\widetilde{y}}}
\def\Ttilde{{\widetilde{T}}}
\def\Mtilde{{\widetilde{M}}}
\def\Vtilde{{\widetilde{V}}}
\def\hatI{{\widehat{I}}}
\def\hatGamma{{\widehat{\Gamma}}}
\def\barR{{\overline{R}}}
\def\hatbarR{\widehat{{\overline{R}}}}
\def\Rtilde{{\widetilde{R}}}
\def\RtildeG{{\widetilde{R}_\Gamma}}
\def\hatRtilde{{\widehat{\widetilde{R}}}}
\def\hatdelta{{\widehat{\delta}}}
\def\hatD{{\widehat{D}}}
\def\hatE{{\widehat{E}}}
\def\hatB{{\widehat{B}}}
\def\hatP{{\widehat{P}}}
\def\hatR{{\widehat{R}}}
\def\hatM{{\widehat{M}}}
\def\hatN{{\widehat{N}}}
\def\hatA{{\widehat{A}}}
\def\hatPi{{\widehat{\Pi}}}
\def\hatH{{\hat{H}}}
\def\hatu{{\hat{u}}}
\def\hata{{\hat{a}}}
\def\hatdelta{{\widehat{\delta}}}
\def\hatS{{\widehat{S}}}
\def\K{\left(\hatRtilde_{\hatGamma}^{T}\Ttilde\hatRtilde_{\hatGamma}\right)}
\def\logHh{\left(1+\log\frac{H}{h}\right)}
\def\q{{{\bf q}}}
\def\vvecV{{{\bf V}}}
\def\vvecn{{{\bf n}}}
\def\P{{{\bf P}}}
\def\r{{{\bf r}}}
\def\etae{{{\bm\eta}}}
\def\mue{{{\bm\mu}}}
\def\q{{{\bm q}}}
\def\lambdae{{{\bm\lambda}}}
\def\vbeta{{{\bm \beta}}}
\def\Lhat{\widehat{{\bm\Lambda}}}
\def\Ltilde{\widetilde{{\bm\Lambda}}}
\newcommand{\EQ}[1]{(\ref{equation:#1})}
\def\beginproof{\indent {\it Proof:~}}
\def\endproof{\qquad $\Box $}
\def\mymax{{\max(\frac{\Ctwo^2}{{\delta}}C_E(\delta,h), C_L(\delta,H,h))}}
\def\SG{{S_\Gamma}}
\def\SGS{{B_\Gamma}}
\def\Lhhat{\widehat{{\bf L}}}
\def\Cone{\nu_{1,\delta,\tau_K,h}}

\def\Ctwo{\nu_{\epsilon,\tau_K,h}}

\def\Cthree{\nu{3,\delta,\tau_K,h}}

\def\C4{{C_{4,\delta,{\bm\beta},\tau_K,h}}}
\def\blambda{{\pmb{\lambda}}}
\def\bLambda{{\pmb{\Lambda}}}
\def\bmu{{\pmb{\mu}}}
\def\bvarphi{{\pmb{\varphi}}}
\def\balpha{{\bm\alpha}}
\def\bxi{{\bm\xi}}

\title{Stochastic BDDC algorithms}
\date{}

\author {Xuemin Tu\thanks{Department of Mathematics, University of Kansas, 1460 Jayhawk Blvd, Lawrence, KS 66045-7594, U.S.A, E-mail:
{xuemin@ku.edu}} 
\and  Jinjin Zhang\thanks{Department of Mathematics, The Ohio State University, Columbus, OH, 43210, U.S.A, E-mail:
{zhang.14647@osu.edu}}
}

\maketitle

\begin{abstract}
Stochastic balancing domain decomposition by constraints (BDDC)
algorithms are developed and analyzed for {{the sampling of the solutions}} of linear stochastic
elliptic equations with random coefficients. Different from the
deterministic BDDC algorithms, the stochastic BDDC algorithms have online and offline
stages. At the offline stage, the Polynomial Chaos (PC) expansions of different components of the BDDC
algorithms are constructed based on the subdomain local
parametrization of the stochastic coefficients. During the online
stage, the sample-dependent BDDC algorithm can be implemented with a
small cost. Under some assumptions, the condition number of the
stochastic BDDC preconditioned operator is estimated. Numerical
experiments confirm the  theory and show that the stochastic BDDC algorithm
outperforms the BDDC preconditioner constructed {{using}} the mean value of the
stochastic coefficients.

\end{abstract}

\section{Introduction}
The solutions of the stochastic partial differential equations (SPDEs) have
wide applications in sciences and engineering such as  the uncertainty
quantifications, Bayesian inferences, and data
assimilations. Efficient and robust solvers for SPDEs  are essential to many
 applications. In this paper, we aim building fast solvers
using domain decomposition algorithms
for sampling methods which involves a class of elliptic PDEs with
random diffusion coefficients.

Two popular classes of methods for
the solutions of SPDEs are  spectral methods and sampling based
methods. Spectral methods build the
functional  approximation of the
solutions on the stochastic coefficients using orthogonal
polynomials \cite{GSSFEM2003,LeMaitre}.
In {these} methods, a discretization
of the random coefficient is introduced using a finite
set of random variables. Many random variables might be needed for
problems with {{complicated uncertainty.}}  The Stochastic Galerkin method,
an  intrusive spectral method,   transforms a stochastic PDE
into a coupled set of deterministic PDEs which lead to large systems of
algebraic equations after  spatial/temporal discretizations. The sampling based methods include Monte Carlo
methods \cite{Liu2008} and non-intrusive collocation stochastic finite
element methods \cite{XHcollocation2005,BNTcollocation2010}. These methods
require a large number of {{solutions} of {SPDEs}} with particular values of the
coefficients which are obtained randomly or
deterministically. Efficient deterministic solvers are very
important.

Domain decomposition methods reduce large problems into collections of
smaller problems by decomposing computational domains into smaller
subdomains. These subdomain problems are computationally easier than
the {{original problems,}} and most of them (if not all) can be solved
independently. {{Therefore},} domain decomposition methods
have  provided  efficient and robust solvers for the systems
arising from discretizations of deterministic partial differential
equations (PDEs) \cite{Toselli:2004:DDM}. The balancing
domain decomposition by constraints (BDDC) methods, one of the most
popular nonoverlapping domain decomposition methods,  were introduced
in \cite{Dohrmann:2003:PSC} and  analyzed in \cite{Jan1,
Jan2} for symmetric positive definite problems. The BDDC methods have
also been extended to solving the  linear systems resulting from the
discretization of different deterministic PDEs \cite{LiBDDCS,Tu:2005:BPP, Tu:2005:BPD,
  Tu_thesis,TW:2016:HDG,TW:2017:WGS,TWZstokes2020,TZAD2021,ZTBrinkman2022,TZOseen2022,STX2021}.

Applications of the domain decomposition methods for PDEs with random
coefficients have attracted {much attention}
recently. {{For example}}, domain decomposition ideas were used with an importance
sampling algorithm in
\cite{LiaoWillcox2015},  with multiscale finite element methods in
\cite{HLZ2017},  with basis adaptations in
\cite{TST2017,Tipireddy2018}, and via moment matching/minimization 
\cite{Cho2015,ZhangBabaeeKarniadakis2018}.  Domain decomposition methods have also
been used to construct reduced models in \cite{MZ2018}.  
Many
preconditioning techniques have been developed for the large linear
systems arising from the Stochastic Galerkin methods
\cite{PE2009,Ullmann2010,PU2010,RV2010,SS2013,SS2014,SGP2014,DKPPS2018,BLY2021}. 
 Among them, different domain decomposition algorithms
including BDDC methods have been applied 
\cite{SBG2009,SS2013,SS2014,DKPPS2018}. 
 For the sample based methods, domain decomposition methods {can be}
applied to solve the deterministic 
system {{for}} each given coefficient of the SPDEs \cite{MTWC}.

In this paper, we construct stochastic BDDC algorithms for the
sampling based methods to speedup each sample simulation. Different
from the deterministic BDDC algorithms which require expensive construction for
each sample, the stochastic BDDC algorithms have online-offline
stages.  The expensive construction is completed at the offline
stage. In the online stage, for each sample, the sample-dependent BDDC
algorithm can be implemented with a small cost. Our stochastic BDDC
algorithms use low-dimensional subdomain local parametrization of the
stochastic coefficient and local Polynomial Chaos (PC)
expansion \cite{Contreras2018B,Contreras2018A}. Compared to those works in
\cite{SBG2009,SS2013,SS2014,DKPPS2018}, which depend on the global
discretization of the random coefficients of the SPDEs,  our constructions of the BDDC
algorithms are local to each subdomain and therefore the 
applications with complex  high-dimensional
uncertainty sources will be feasible.  Exploring local parametrization
for the sampling based methods have been studied in
\cite{ChenJGX2015,Contreras2018A,Contreras2018B,stochasticpRMCM2021,StochasticFETIDP2021}. Particularly in
\cite{Contreras2018B}, a stochastic subdomain interface global Schur
complement is constructed using subdomain local parametrization in the
offline stage.  In the online stage, the sample-based Schur complement
is approximated by  the stochastic {surrogate} Schur complement with a small cost.
The convergence of this approximation to the exact Schur complement
has {{been}} demonstrated by numerical experiments.  In
\cite{stochasticpRMCM2021},  these sampled-based approximated Schur complements are
used as preconditioners for solving exact Schur complement for each
sample.
For large scale high dimensional applications, {a} large number
of subdomains can   give small subdomain local
problem size and therefore a
small number of  subdomain local random variables in the local
parametrization. These will help to reduce the computational
complexity and memory requirements of the offline stage \cite{Contreras2018B}. However, the size of the global Schur complement
increases with the number of the subdomains. To form the global Schur
complement explicitly and use direct solvers for the
solution will be very  expensive if still possible and this leads to a
limitation of the works in
\cite{ChenJGX2015,Contreras2018B,stochasticpRMCM2021}.
Another approach proposed in \cite{Contreras2018B } is that the
global Schur complement is not formed explicitly and  the matrix-vector
multiplication is calculated in each subdomain for the conjugate gradient (CG)
iterations.  However, no
preconditioners are investigated for this approach and the number of
CG iterations needed for a certain accuracy might be large.
The stochastic BDDC algorithms we proposed in this paper can provide
efficient preconditioners for this approach. 
Similar to  the
deterministic BDDC algorithms \cite{Widlund:2020:BDDC},  the
stochastic BDDC algorithm neither {{forms}} the global Schur complement
explicitly nor {{requires}} the exact/approximate of factorizations {{of}} the global
Schur complement. In the online stage, only a small coarse problem is
constructed. The size of the {{coarse}} problem is proportional to the
number of the subdomains, which is  much smaller than the global Schur
complement. When the {number} of the {subdomains} is extremely large, the
coarse problem might become a bottle-neck and
three or multilevel strategies used in the deterministic 
BDDC algorithms,
\cite{Tu:2004:TLB,Tu:2005:TLB,Tu:2008:TBS,ZT:2016:Darcy}, can be
applied similarly  in our stochastic BDDC algorithms  to remove the
bottle-neck. Under some assumptions, we analyze the condition numbers of our stochastic BDDC
algorithms and   compare the performance  with the
deterministic (exact)  BDDC algorithms and the algorithms with the preconditioner
constructed 
using the mean value of the samples. Our stochastic BDDC algorithms
outperform the algorithms based
on the mean value and have similar performance as the exact BDDC
algorithms,  but the
expensive construction of the preconditioners is removed at the
online stage.  


The rest of the paper is organized as follows. We describe the
stochastic PDEs and sampling methods in Section 2. In Section 3, the deterministic
BDDC algorithms  is given  and the detailed stochastic BDDC algorithms are
provided in Section 4.  In Section 5, {the} analysis of our stochastic
BDDC algorithms are provided. Finally, we present some computational results in Section 6.

\section{ A stochastic elliptic equation and a  finite element discretization}
Let ${\cal P}=(\Theta,\Sigma_\Theta,{P})$ be a probability
space, where  $\Theta$ is the set of random events, $\Sigma_\Theta$ is
the associated $\sigma$-algebra, and {$P$} is the probability
measure.
Denote $\langle\cdot\rangle$ as the expectation operator with
the probability measure ${P}$  for any random variables $\xi$
defined on the space$(\Theta, \Sigma_\Theta,{P})$.  {In this paper, we also use $E(\xi)$ to present the
  expectation of $\xi$.}
\begin{align}\label{equation:expectation}
E(\xi)=\langle\xi\rangle=\int_{\Theta}\xi(\theta)  d{P(\theta).}
\end{align}

We consider the stochastic elliptic equation on a two-dimensional domain $\Omega$
\begin{equation}\left\{
\label{equation:pde1}
\begin{array}{rcl}
-\nabla\cdot\left(\kappa (\x,\theta)\nabla{u}(\x,\theta)\right)&=&f(\x),\quad
                                                        \mbox{in}\quad
                                                        \Omega\times \Theta,\\
u(\x,\theta)&=&{0},\quad\mbox{on}\quad\partial
                                    \Omega\times \Theta,
\end{array}\right.
\end{equation}
where $f(\x)$ is a deterministic function in $L^2(\Omega)$. We
assume that  the
stochastic diffusion coefficient $\kappa
(\x,\theta)=\exp(a(\x,\theta))$, where $a(\x,\theta)$ is a centered Gaussian field with covariance function $C$
\begin{align*}
a(\x,\theta)\sim N(0,C).
\end{align*}
\EQ{pde1} is well posed under mild conditions \cite{CPDE2012} and we
have {{$u\in L^2(\Theta;H^1_0(\Omega))$}} with {{$u(\cdot,\theta)\in H^1_0(\Omega)$ for a.e. $\theta\in\Theta$}}. 
Similarly  as in\cite{stochasticpRMCM2021}, we take 
$C(\x,\y)$ as
\begin{align}\label{equation:cov}
C(\x,\y)=\sigma^2\exp(-\frac{\|\x-\y\|_2^2}{l}),
\end{align}
with the correlation length $l>0$. 


In sampling based methods, such as Monte Carlo
methods and non-intrusive collocation stochastic finite element
methods, a large number of the solutions of \EQ{pde1}, with particular
values $\theta$ and $\kappa(x,\theta)$, are
required. Therefore, fast solvers for each $\theta$ is crucial  for the
computation.   In this paper, we will introduce domain decomposition
based solvers for this purpose.

The rest of this section, we introduce the finite element
discretization of \EQ{pde1} with a given $\theta$.
Let $\T_h$ be a shape-regular and quasi-uniform triangulation of
$\Omega$ and the element in $\T_h$ is denoted by $K$.  $P_1(K)$ is the
space of polynomials of order at most $1$. 
We define a piece-wise linear finite element space $\What$ as:
{{$$\What=\{\phi(\x)\in H^1_0(\Omega):\phi(\x)|_K\in P_1(K),\forall K\in\T_h\}.$$}}
With a standard finite element procedure \cite{Braess:1997:FET}, the solution of \eqref{equation:pde1} can be approximated as:
\begin{align*}
u(\x,\theta)\approx\sum\limits_{i=1}^{N_x} \phi_i(\x)u_{i}(\theta),
\end{align*}
where $N_x$ is the dimension of $\What$ and $u(\theta)=\left(u_1(\theta),\cdots,u_{N_x}(\theta)\right)^T$ is the solution of the following linear system
\begin{equation}\label{equation:geqA}
A(\theta)u(\theta)=F.
\end{equation}
Here $A$ is the stiffness matrix and $F$ is the right hand side.  
{$A$ and $F$ can be formed by 
\begin{equation}\label{equation:Aij}
 { {A_{s,t}(\theta)=\int_{\Omega}\kappa(\x,\theta)\nabla\phi_s(\x)\nabla\phi_t(\x)d\x,\quad
  F_{s}=\int_{\Omega}\phi_s(\x)f(\x)d\x.}}
\end{equation}
}

When $N_x$ is large,  solving the linear system \EQ{geqA} might be very
expensive. We will introduce deterministic BDDC algorithms for
solving \EQ{geqA} in next section.  

\section{Domain decomposition and deterministic BDDC algorithms}

 We decompose the original computational domain $\Omega$ 
into $N$ nonoverlapping polyhedral subdomains $\Omega^{(i)}$ without
cutting any element in $\T_h$. 
We assume
that  each
subdomain is a union of shape regular coarse elements.  Let $H$ be the typical diameter
of the subdomains and  $\Gamma
= {(\cup\partial\Omega^{(i)})} \backslash
\partial\Omega$ be the subdomain interface shared by neighboring
subdomains.
The interface of
subdomain $\Omega^{(i)}$ is denoted by $\Gamma_i = \partial \Omega^{(i)}
\cap \Gamma$.  We note that for the less regular subdomains obtained
from mesh partitioners, our algorithm is also defined and the analysis with irregular
subdomains in domain decomposition methods, see \cite{DKW2008O,KRW2008F,DKWDD22,WDD23,DW2012}.

We first reduce the global system \EQ{geqA} into a subdomain interface
problem on $\Gamma$. 
In order to do that, we decompose 
the space $\What$ as follows:
\[\What=W_I \oplus \What_\Gamma= \left(\Pi_{i=1}^N W_I^{(i)}\right)\oplus
  \What_\Gamma,\]
where
$W^{(i)}_I$  are the spaces of the subdomain interior variables, while
$\What_\Gamma$ is the subspace corresponding to the variables on the
interface. {{Throughout the paper, we use the same symbol for a finite element function and its degrees of freedom vector.}} 
{For a fixed $\theta$}, we can rewrite the original  problem~\EQ{geqA} 
as: find $u_I (\theta)\in W_I$ and $u_\Gamma(\theta) \in \What_\Gamma$, such that
\begin{equation}
\label{equation:twoparts} \left[
\begin{array}{cc}
A_{II}(\theta)       &  A_{ \Gamma I}^T(\theta)     \\
A_{\Gamma I}(\theta) & A_{\Gamma \Gamma}(\theta)
\end{array}
\right] \left[
\begin{array}{c}
u_I(\theta)        \\
u_\Gamma(\theta)
\end{array}
\right] = \left[
\begin{array}{c}
f_I        \\
f_\Gamma    \\
\end{array}
\right].
\end{equation}
Here 
$A_{\Gamma I}(\theta)$ and $A_{\Gamma \Gamma}(\theta)$ are assembled from subdomain matrices across the
subdomain interfaces. $A_{II}(\theta)$ is block diagonal and each
block corresponds to  one subdomain. Therefore, we can 
eliminate the subdomain interior variables $u_I(\theta)$ in each subdomain independently
from~\EQ{twoparts} and reduce the original system \EQ{geqA} into a subdomain interface
problem.  The subdomain local Schur complement
$S_\Gamma^{(i)}(\theta)$ is defined as
\begin{equation}
\label{equation:dissubschur} S_\Gamma^{(i)}(\theta)=A^{(i)}_{\Gamma \Gamma} (\theta)-A^{(i)}_{\Gamma I}
(\theta)A_{II}^{(i)^{-1}}(\theta)A^{(i)^T}_{\Gamma I} (\theta),
\end{equation}
where $A^{(i)}_{\Gamma \Gamma}(\theta)$, $A^{(i)}_{\Gamma I}(\theta)$, and $A_{II}^{(i)}(\theta)$ are
subdomain local matrices. 
The global Schur complement $S_\Gamma(\theta)$ can be assembled from the
subdomain Schur complement $S^{(i)}_\Gamma(\theta)$  and the global
interface problem is defined as:
find {{$u_\Gamma(\theta) \in \What_\Gamma$}} such that 
\begin{equation}
\label{equation:globalschur} 
  S_\Gamma (\theta) u_\Gamma(\theta) = g_\Gamma(\theta) ,
  \end{equation}
  where
$g_\Gamma(\theta) =f_\Gamma-A_{\Gamma I}(\theta)A_{II}(\theta) ^{-1}f_I$.

 In order to introduce the BDDC preconditioner, we
further decompose $\What_\Gamma$ into the primal interface variables 
and the remaining (dual) variables. We denote the primal variable space as
$\What_\Pi$ and the dual variable space $W_\Delta$.  We relax the
continuity for the dual variables and introduce a partially assembled
interface space as
\[
  \Wtilde_\Gamma=\What_\Pi\oplus W_\Delta=\What_\Pi\oplus
  \left(\Pi_{i=1}^N W_\Delta^{(i)}\right).
\]
Here the degrees of freedom in $ W_\Delta$ may be
discontinuous across the subdomain interface. The subspace $\What_\Pi$
contains the coarse level, continuous primal
interface degrees of freedom. 
Define a 
 partially sub-assembled
 problem matrix $\Atilde(\theta)$ as a two by two block form
\begin{equation}
\label{equation:anothertwoparts} \Atilde(\theta)=\left[
\begin{array}{cc}
A_{II}(\theta)       &  \Atilde^T_{\Gamma I} (\theta)     \\
\Atilde_{\Gamma I} (\theta) & \Atilde_{\Gamma \Gamma}(\theta) 
\end{array}
\right],
\end{equation}
where
\[\Atilde_{I\Gamma}(\theta) =\left[\begin{matrix} A_{I\Delta} (\theta) &
      A_{I\Pi}(\theta) \end{matrix}\right],
  \quad \Atilde_{\Gamma\Gamma}(\theta) =\left[\begin{matrix} A_{\Delta\Delta}(\theta)  &
      {A^T_{\Pi\Delta}(\theta)} \\A_{\Pi\Delta}(\theta) & A_{\Pi\Pi}(\theta) \end{matrix}\right].\]
We note that $\Atilde_{\Gamma I}(\theta)$ and  $\Atilde_{\Gamma \Gamma}(\theta) $ are assembled
only for the coarse level primal degrees of freedom 
across the interface. We define the partially sub-assembled Schur complement
operator $\widetilde{S}_\Gamma$ as
\begin{equation}\label{equation:schurT}
  \widetilde{S}_\Gamma(\theta) 
= \Atilde_{\Gamma \Gamma} (\theta) - {\Atilde_{\Gamma I} (\theta)  A_{II}^{-1}(\theta) 
\Atilde^T_{\Gamma I}(\theta)} .
\end{equation}
We can obtain $\widetilde{S}_\Gamma(\theta) $ by partially
assembling the subdomain local Schur complement $S^{(i)}_\Gamma(\theta) $
defined in \EQ{dissubschur} with respect to the primal interface
variables. Denote the injection operator from {{$\What_\Gamma$
to $\Wtilde_\Gamma$ by $\Rtilde_\Gamma$.}} We can   further
assemble  $\Stilde_\Gamma(\theta)$ with respect to the dual
interface variables to obtain $S_\Gamma(\theta)$, i.e.,
\[
S_\Gamma(\theta) =\Rtilde_\Gamma^T \Stilde_\Gamma(\theta)  \Rtilde_\Gamma.
\]

Let $W_r^{(i)}=W_I^{(i)}\oplus W_\Delta^{(i)}$  and
\begin{equation}\label{equation:Arc}
  A_{rr}^{(i)}=\left[ \begin{array}{cc} A^{(i)}_{II}(\theta) 
& {A^{(i)^T}_{\Delta I}(\theta) } \\ A^{(i)}_{\Delta I}(\theta)  &  A^{(i)}_{\Delta \Delta}(\theta)  \\
                      \end{array} \right],\quad A^{(i)}_{cr}=\left[ A^{(i)}_{\Pi I}(\theta)  \quad A^{(i)}_{\Pi \Delta}(\theta) 
                    \right],\quad A_{cc}^{(i)}= A^{(i)}_{\Pi \Pi} (\theta).
                  \end{equation}
 We can rewrite \EQ{anothertwoparts} as                  
 \begin{equation}
\label{equation:rctwoparts} \Atilde(\theta)=\left[
\begin{array}{cc}
\Atilde_{rr}(\theta)       &  \Atilde^T_{cr} (\theta)     \\
\Atilde_{cr} (\theta) & A_{cc}(\theta) 
\end{array}
\right],
\end{equation}
where
$\Atilde_{rr}(\theta)$ is a block diagonal matrix with $A_{rr}^{(i)}$
on the diagonal.  $\Atilde_{cr}(\theta)$  is assembled from $A_{cr}^{(i)}$
only for the coarse level primal variables $\Pi$ and $A_{cc}$ is fully
assembled from $A_{cc}^{(i)}$.

We define $
\Rtilde_{D,\Gamma}(\theta)=D(\theta) \Rtilde_\Gamma$, where $D(\theta)$ is a 
scaling matrix and should provide a partition of unity:
\begin{equation}\label{equation:PU}\widetilde{R}_{D,\Gamma}(\theta)^T
\widetilde{R}_{\Gamma}=\widetilde{R}_{\Gamma}^T
\widetilde{R}_{D,\Gamma}(\theta)=I.
\end{equation}
Different choices of the scaling matrix $D$ can be found in
\cite{ZT:2016:Darcy,Widlund:2020:BDDC}. {{Define $R_{\Gamma\Pi}$  as the mapping from $\widetilde{W}_\Gamma$ onto its subspace $W_{\Pi}$ and $R^{(i)}_{\Gamma\Pi}$ as the restriction of $R_{\Gamma\Pi}$ to the local interface of subdomain $\Omega^{(i)}$.   }}
 The BDDC preconditioned interface problem is
\begin{equation}
\label{equation:preconditioned}
\Rtilde^T_{D,\Gamma}(\theta)\Stilde_{\Gamma}^{-1}(\theta) \Rtilde_{D,\Gamma}(\theta)
S_{\Gamma}(\theta) u_{\Gamma}(\theta) =
\Rtilde^T_{D,\Gamma} (\theta)\Stilde_{\Gamma}^{-1}(\theta) \Rtilde_{D,\Gamma}(\theta)
g_{\Gamma}(\theta),
\end{equation}
where 
\begin{equation}
\label{equation:StInverse}
 \Stilde_\Gamma^{-1}(\theta) =
R_{\Gamma\Delta}^T\left(\sum_{i=1}^{N}
 \left[ \vvec{0} ~ R^{(i)^T}_{\Delta}
 \right] A_{rr}^{{(i)}^{-1}}(\theta)  \left[ \begin{array}{c} \vvec{0}  \\
R^{(i)}_{\Delta} \end{array} \right] \right)R_{\Gamma \Delta} + \Phi(\theta)
S_{\Pi}^{-1} (\theta) \Phi^T(\theta),
\end{equation}
\begin{equation}
\label{equation:phi}
 \Phi (\theta)= R_{\Gamma \Pi}^T - R^T_{\Gamma \Delta}\sum_{i=1}^{N}
  \left[ \vvec{0} ~ R^{(i)^T}_{\Delta}
 \right]A_{rr}^{{(i)}^{-1}}(\theta)  A_{cr}^{(i)^T}(\theta) 
R^{(i)}_{\Pi},\\ 
\end{equation}
and
\begin{equation}
\label{equation:Spi0}
\begin{array}{rcl}
S_{\Pi}(\theta) 
  =  \sum_{i=1}^{N} R^{(i)^T}_{\Pi} \left\{
 A^{(i)}_{cc} (\theta) - A_{cr}^{(i)}(\theta)  A_{rr}^{(i)^{-1}}(\theta)  A_{cr}^{(i)^T}(\theta) \right\} R^{(i)}_{\Pi}.
\end{array}
\end{equation}
{Since $S_\Gamma(\theta)$ and
$\Rtilde^T_{D,\Gamma}(\theta)\Stilde_{\Gamma}^{-1}(\theta)
\Rtilde_{D,\Gamma}(\theta)$ in \EQ{preconditioned} are both   symmetric positive definite, we will use
the conjugate gradient method (CG)  to solve \EQ{preconditioned} with
$\Rtilde^T_{D,\Gamma}(\theta)\Stilde_{\Gamma}^{-1}(\theta)
\Rtilde_{D,\Gamma}(\theta)$ as the preconditioner.}  In practice, we do not form
$S_\Gamma(\theta)$ explicitly. Instead, we will store $A_{\Gamma
  \Gamma}^{(i)}(\theta)$, $A_{\Gamma
  I}^{(i)}(\theta)$, and the Cholesky factors of $A_{II}^{(i)}(\theta)$ in each
subdomain for the matrix-vector multiplication of $S_\Gamma$. For the
preconditioner, we store $A_{cr}^{(i)}(\theta)$ and the Cholesky factors of $A_{rr}^{(i)}(\theta)$ in each
subdomain. We need to form the global coarse matrix $S_\Pi(\theta)$, defined
in \EQ{Spi0}, and store
the Cholesky factor of it as well.

It is well-known \cite{Widlund:2020:BDDC} that, the performance of the BDDC algorithms is mainly 
determined by the choices of the primal variable space $\What_\Pi$ and
the scaling operator $D(\theta)$. The smallest eigenvalue of the preconditioned
BDDC operator is bounded from below by $1$ and the largest eigenvalue
is bounded from above by the bound of an average operator
\begin{equation}\label{equation:ED}
  E_D(\theta)=\Rtilde_\Gamma\Rtilde_{D,\Gamma}(\theta),
\end{equation}
which computes a weighted
average across the subdomain interface and distributes the average
back to each subdomain.

For two dimensional problems \cite{Widlund:2020:BDDC}, if the coefficient $\kappa$ in \EQ{pde1} is
a constant or has a small variation  in each subdomain,  the condition
number of the preconditioned BDDC operator in \EQ{preconditioned}  is
bounded by $C\left(1+\log\frac{H}{h}\right)^2$ when $\What_\Pi$
includes the vertices of each subdomain and $D$ is a simple
$\rho$-scaling. Here $H$ is the size of the subdomain and $h$ is the
size of the mesh, $C$ is a constant independent of $H$, $h$,
$\kappa$.  When $\kappa$ has a large variation in each subdomain, some
adaptive primal variables have to be included in $\What_\Pi$ and the
deluxe scaling is needed to ensure the good performance of the BDDC
algorithms. {In this paper, we only consider $\What_\Pi$ including the vertices of each subdomain. The adaptive primal constrains will be considered in our future work. }

\section{Stochastic BDDC algorithms}
In sampling based methods,  one needs to compute  {{a large number of
solutions}} of \EQ{pde1} with different values of $\theta$. For a given
$\theta$, one can follow the procedure discussed in {{the previous sections}}
to solve \EQ{preconditioned} using CG.  In
order to do that, as discussed at the end of {{the last section}}, we need to
form several subdomain local matrices and a global coarse matrix.
The construction of these matrices will require
{an} expensive set-up procedure,
which makes the overall algorithms less efficient.
There are some  deterministic preconditioner approaches\cite{PE2009}, {which} form 
preconditioners using {a particular value of $\theta$}
 or the mean value of $\theta$ over all the samples, and then apply the  same  preconditioner to
 each sample.  When the stochastic parameter has high variability,
 the  deterministic preconditioner might be inefficient.

 In this section, we
will focus on how we can use the subdomain local parametrization and
{{PC} expansion} to approximate the BDDC preconditioner. We construct
the {{PC}} expansion of different components of the BDDC
preconditioner at the offline stage. In the online stage,  given a 
$\theta$, with a
small cost, we can form a $\theta$-based preconditioner to solve
\EQ{preconditioned}.  We call this preconditioner  the stochastic BDDC
preconditioner. This preconditioner can be used for {{solving the exact}} 
$S_\Gamma$  or an approximate $S_\Gamma$ proposed in
\cite{Contreras2018B}. It is also effective for the stochastic system
with high variability.

{ To make our notation simple, we will construct our
  stochastic BDDC preconditioner for single-level Monte Carlo
  methods.  There are active research on multi-level Monte Carlo
  methods
  \cite{barth_multi-level_2011,Cliffe2011,GilesActa2015,ali_multilevel_2017,giles2018book,dodwell_multilevel_2019},
  which can significantly speed up the single-level methods. For those
 multi-level methods based on different levels of meshes such as \cite{barth_multi-level_2011,Cliffe2011}, we can
  similarly construct our stochastic BDDC preconditioners at different
  levels to
  precondition  the
  solver at that level. 
  Moreover,
  the coarse problems built  in the BDDC algorithms are closely related to
the system on the corresponding coarse mesh, see \cite[Lemma
4.2]{Tu:2004:TLB,Tu:2005:TLB}. In our future work, we will study how the coarse problem
can be used in the multi-level Monte Carlo methods and make
the connection {{between}} the
multi-level BDDC \cite{MandelML,ZT:2016:Darcy} and multi-level Monte Carlo
algorithms. }

In the following subsections, we first {{introduce}}
the subdomain local Karhunen-Lo$\grave{e}$ve (KL) expansion and the
{{PC}}
approximation, {{as discussed in \cite{Contreras2018B}}}. We then
construct our stochastic BDDC preconditioners using both stochastic
collocation and Galerkin methods. 

\subsection{Subdomain local KL expansions and {{PC}} approximations}
The KL expansion of $a(\x,\theta)=\log \kappa(\x,\theta)$ can be represented as
\begin{align*}
{a}(\x,\theta)=\sum\limits_{m=1}^{\infty} {\sqrt{\lambda_m}}a_m(\x)\xi_m(\theta),
\end{align*}
where $(\lambda_m,a_m(\x))^{\infty}_{i=1}$ are the set of eigenvalues
and {corresponding} eigenfunctions of $C(\x,\y)$, defined in \EQ{cov}, with
$\lambda_1>\lambda_2>\cdots$. The random variables
$\{\xi_1,\xi_2,\cdots\}$ are independent, identical distributed, and $\xi_m \sim N(0,1).$ 
The {global} truncated KL expansion of $a(\x,\theta)$ is denoted as
\begin{equation}\label{equation:globalKL}
a(\x,\theta)\approx {a_\MK}(\x,\theta)=\sum\limits_{m=1}^{M_{KL}}\sqrt{\lambda_m}a_m(\x)\xi_m(\theta),
\end{equation}
where the $\lambda_m$  and $a_m({\x})$ are the dominant $M_{KL}$ eigenvalues and corresponding eigenfunctions of $C(\x,\y).$
$a_{M_{KL}}(\x,\theta)$ is an approximation of $a(\x,\theta)$ with a reduced dimension.
{Given an $M_{KL}$, the approximation error of $a_{M_{KL}}(\x,\theta)$
can be described 
as {\cite{Contreras2018A}}:  }
\begin{equation}\label{equation:globalKLerror}
  E(\|a_\MK(\x,\theta)-a(\x,\theta)\|^2_{2})=\sum\limits_{m=M_{KL}+1}^{\infty}\lambda_m{{\|a_m\|^2_{2}}}.
  \end{equation}

{Denote $a^{(i)}(\x,\theta)$ as the restriction of $a(\x,\theta)$ to
the subdomain $\Omega^{(i)}$ and let $C^{(i)}(\x,\y)$
  be  the restriction of $C(\x,\y)$ to
the subdomain $\Omega^{(i)}$. $(\lambda^{(i)}_m,a^{(i)}_m)$ are the set of eigenvalues and {corresponding}
eigenfunctions of $C^{(i)}(\x,\y)$, namely as \cite[Equation (3.3)]{ChenJGX2015},
$$\int_{\Omega^{(i)}} C^{(i)}(\x,\y) a^{(i)}_m(\x)d\x=\int_{\Omega^{(i)}} C(\x,\y) a^{(i)}_m(\x)d\x =\lambda^{(i)}_m a^{(i)}_m(\y).$$}
The KL expansion of $a^{(i)}(\x,\theta)$ in $\Omega^{(i)}$ can be represented as
\begin{equation}\label{equation:LML}
{a}^{(i)}(\x,\theta)=\sum\limits_{m=1}^{\infty}\sqrt{\lambda^{(i)}_m}a^{(i)}_{m}(\x)\xi^{(i)}_m(\theta).
\end{equation}
We label the eigenvalues in descending order $\lambda^{(i)}_1>\lambda^{(i)}_2>\cdots.$ 
The local truncated KL expansion of $a^{(i)}(\x,\theta)$ can be
represented as: 
\begin{equation}\label{equation:aKL}
a^{(i)}(\x,\theta)\approx {a}^{(i)}_{M_{KL}}(\x,\theta)=\sum\limits_{m=1}^{{M^{(i)}_{KL}}}\sqrt{\lambda^{(i)}_m}{a}^{(i)}_m(\x)\xi^{(i)}_m(\theta),
\end{equation}
where $M^{(i)}_{KL}$ is the number of terms  kept in the expansion.
{The local KL truncation error can be obtained as
\begin{equation}\label{equation:localerror}
E(\|a^{(i)}(\x,\theta)-a^{(i)}_{M_{KL}}(\x,\theta)\|^2_2)=\sum\limits_{m=M^{(i)}_{KL}+1}^{\infty}\lambda^{(i)}_m
\|a^{{(i)}}_m\|^2_2.
\end{equation}

{Since  $\{a_{m}^{(i)}\}$ forms an
  orthonormal basis, similar to \cite[Equation (3.19)]{ChenJGX2015},
  given a sample $\theta$,  we can obtain the subdomain local 
$\xi^{(i)}_m(\theta)$ in \EQ{LML} for 
$a^{(i)}(\x,\theta)$ as follows: for $i=1,2,\cdots$, }
\begin{equation}\label{equation:thetaprojection}
\begin{aligned}
\xi^{(i)}_m(\theta)&=\frac1{\sqrt{\lambda^{(i)}_m}}\int_{\Omega^{(i)}}
a(\x,\theta){a}^{(i)}_m(\x)d\x=\frac1{\sqrt{\lambda^{(i)}_m}}\int_{\Omega^{(i)}}
\left(\sum\limits_{n=1}^{\infty}\sqrt{\lambda_n}{a}_n(\x)\xi_n(\theta)\right)a^{(i)}_m(\x)d\x.
\end{aligned}
\end{equation}

{ Given $a_{M_{KL}}(\x,\theta)$,
let $\hat{\xi}^{(i)}(\theta)=\left(\hat{\xi}^{(i)}_1(\theta),\cdots,
\hat{\xi}^{(i)}_{M_{KL}^{(i)}}(\theta)\right)$, which can  be obtained by solving a least square
problem in each subdomain $\Omega^{(i)}$:
\begin{equation}\label{equation:localxi}
\hat{\xi}^{(i)}(\theta)=\underset{(v_1,\cdots,v_m,\cdots,v_{M_{KL}^{(i)}})\in \mathbb{R}^{M_{KL}^{(i)}}}{\arg\min}\int_{\Omega^{(i)}}\left(\exp\left(a_{M_{KL}}(\x,\theta)\right)
  -\exp\left(\sum\limits_{m=1}^{{M^{(i)}_{KL}}}\sqrt{\lambda^{(i)}_m}{a}^{(i)}_m(\x)v_m\right)\right)^2d\x.
\end{equation}
We can define the truncated local KL expansion $a_{M_{KL}}^{(i)}=\sum\limits_{m=1}^{{M^{(i)}_{KL}}}\sqrt{\lambda^{(i)}_m}{a}^{(i)}_m(\x)\hat{\xi}^{(i)}_m(\theta)$.
Note that any continuity of $\exp\left(a_{M_{KL}}(\x,\theta)\right)$ at
the subdomain 
interface can be built in as a constraint, \cite{ChenJGX2015}. }

{When $M_{KL}^{(i)}$ is large such that $a_{M_{KL}}(\x,\theta)$ can be well
presented by the subspace spanned by
$\{a_m^{(i)}\}_{m=1}^{M_{KL}^{(i)}}$, \EQ{localxi} can be approximated
by the following
 least square problem, which are used in for example
 \cite{MZ2018,stochasticpRMCM2021},  
\begin{equation}\label{equation:alocalxi}
\hat{\xi}^{(i)}(\theta)\approx \underset{(v_1,\cdots,v_m,\cdots,v_{M_{KL}^{(i)}})\in \mathbb{R}^{M_{KL}^{(i)}}}{\arg\min}\int_{\Omega^{(i)}}\left(a_{M_{KL}}(\x,\theta)
  -\left(\sum\limits_{m=1}^{{M^{(i)}_{KL}}}\sqrt{\lambda^{(i)}_m}{a}^{(i)}_m(\x)v_m\right)\right)^2d\x.
\end{equation} The solution of the  least square problem \EQ{alocalxi} is given by 
  \begin{equation}\label{equation:lxi}  
\hat{\xi}^{(i)}_m(\theta)=\frac1{\sqrt{\lambda^{(i)}_m}}\int_{\Omega^{(i)}}
a_{M_{KL}}(\x,\theta){a}^{(i)}_m(\x)d\x=\frac1{\sqrt{\lambda^{(i)}_m}}\int_{\Omega^{(i)}}
\left(\sum\limits_{n=1}^{{{M}_{KL}}}\sqrt{\lambda_n}{a}_n(\x)\xi_n(\theta)\right)a^{(i)}_m(\x)d\x.
\end{equation}}

The subdomain local matrices $A_{II}^{(i)}(\theta)$, $A_{\Gamma I}^{(i)}(\theta)$, $A^{(i)}_{\Gamma\Gamma}(\theta)$,  $A_{cr}^{(i)}(\theta)$, $A_{rr}^{(i)}(\theta)$,
and $A_{cc}^{(i)}(\theta)$,  required in the BDDC algorithm
\EQ{preconditioned}, depend on the subdomain local
$a^{(i)}(\x,\theta)$ only. They can be considered as functionals of
$\xi^{(i)}_m(\theta) $ and truncated spectral {expansions} can be used to
approximate these matrices. 
To obtain a similar accuracy,
$M_{KL}^{(i)}$, the number of the subdomain local KL
expansion terms, can be  much smaller than the global $\MK$.  Our BDDC
algorithm is constructed using local KL expansion instead of the
global KL expansion. This is a big difference compared with the previous
domain decomposition algorithms for solving \EQ{pde1} as in \cite{SBG2009,SS2013,SS2014,DKPPS2018}.


In the subdomain $\Omega^{(i)}$, let $\psi^{(i)}_{\alpha_l}$ be the univariate Hermite
polynomial of degree $\alpha_l$ and $\psi^{(i)}_{\bm\alpha}$ be the
products of those orthonormal  univariate Hermite polynomials 
as
\begin{align*}
\psi^{(i)}_{\bm\alpha}({\bm\xi})=\prod\limits_{l=1}^{M^{(i)}_{KL}}\psi^{(i)}_{\alpha_l}(\xi_l),
\end{align*}
where
${\bm\alpha}=(\alpha_1,\cdots,\alpha_{M^{(i)}_{KL}})\in{\mathscr{S}^{(i)}}={\mathbb{N}^{M^{(i)}_{KL}}_0},$
which is a multi-index set with $M^{(i)}_{KL}$ components and
$|{\bm\alpha}|=\alpha_1+\alpha_2+\cdots+\alpha_{M^{(i)}_{KL}}$.
By \cite{PC1938}, any function $f\in L^2_{\bm\xi}$ has a Polynomial
Chaos (PC)  expansion
as
\begin{equation}\label{equation:gPC}
  f(\bm\xi)=\sum_{\bm\alpha\in{\mathscr{S}^{(i)}}}f_{\bm\alpha}\psi^{(i)}_{\bm\alpha}({\bm\xi}).
  \end{equation}
 In our computation, we truncate the series in \EQ{gPC} to a finite
 number terms as
\begin{equation}\label{equation:gPCA}
  f(\bm\xi)\approx f_{PC_d} (\bm\xi)=\sum_{\bm\alpha\in{\mathscr{S}^{(i)}_d}}f_{\bm\alpha}\psi^{(i)}_{\bm\alpha}({\bm\xi}),
  \end{equation}
where $d$ is a given nonnegative integer and  the finite set of
multi-index $\mathscr{S}^{(i)}_d$ is defined as
\begin{equation}\label{equation:sd}\mathscr{S}^{(i)}_d=\{{\bm\alpha}\in
  \mathbb{N}^{M^{(i)}_{KL}}_0:|{\bm\alpha}|\leq d\}.
 \end{equation}
$n^{(i)}_{\xi}$, the dimension of the space 
$\mathscr{S}^{(i)}_d$, equals to 
$\begin{pmatrix}
M^{(i)}_{KL}+d\\
d
\end{pmatrix}$.

Let  $A^{(i)}_L(\theta)$ be one of the subdomain local matrices. We
will 
approximate it by a truncated PC expansion
as follows 
\begin{align}\label{equation:pc}
A^{(i)}_{L}(\theta)\approx {A^{(i)}_{L,PC_d}}(\theta)=
\sum\limits_{{\bm\alpha}\in {\mathscr{S}}^{(i)}_d}
A^{(i)}_{L,{\bm\alpha}}\psi_{{\bm\alpha}}(\theta),
\end{align}
where $A^{(i)}_{L,{\bm\alpha}}=\langle
A^{(i)}_{L}\psi_{\bm\alpha}\rangle$ is the coefficient of the
PC  expansion with Hermite polynomial
$\psi_{\bm\alpha}$. Recall $\langle
\cdot\rangle$ is the expectation defined in \EQ{expectation}.  There are several ways to estimate the
coefficients $A^{(i)}_{L,{\bm\alpha}} $. The computation of 
$A^{(i)}_{L,{\bm\alpha}}$ for $A_{II}^{(i)^{-1}}$ and
$A_{\Gamma I}^{(i)}$ have been studied using stochastic Galerkin (SG)
in \cite{Contreras2018B} and using stochastic collocation (SC) in
\cite{stochasticpRMCM2021}.  In next subsection, we will discuss in
details about how we can compute the coefficients in \EQ{pc} for those
subdomain local matrices related to the BDDC preconditioners using
both SG and SC methods. 

 \subsection{Stochastic BDDC Preconditioner}
 For the BDDC preconditioners, we need to approximate  subdomain local
 $A_{cr}^{(i)}(\theta)$,  the Cholesky factors of
 $A_{rr}^{(i)}(\theta)$, and
$A_{cc}^{(i)}(\theta)$, namely we need to calculate the PC
coefficients in \EQ{pc} for these matrices.  We will construct these
coefficients using both SG and SC methods.

\subsubsection{SG} \label{sec:SG}
Given any
{$\left\{\xi_m^{(i)}\right\}_{i=1}^{M_{KL}^{(i)}}$,
  $a^{(i)}_{M_{KL}}(\x,\theta))=\sum\limits_{m=1}^{M_{KL}^{(i)}}\sqrt{\lambda^{(i)}_m}a^{(i)}_{m}(\x)\xi^{(i)}_m(\theta)$. }
  We consider the PC expansion of
$\exp(a^{(i)}_{M_{KL}}(\x,\theta))$ 
\begin{align}\label{equation:expaKL}
\exp(a^{(i)}_{M_{KL}})&=\sum\limits_{{\bm\alpha}\in \mathscr{S}^{(i)}}
a_\balpha(x)\psi^{(i)}_{\bm\alpha},
\end{align}
where 
$a_\balpha(x)=\langle \exp(a^{(i)}_{M_{KL}})\psi^{{(i)}}_{\bm\alpha}\rangle$. 

Plugging \EQ{expaKL} in \EQ{Aij} and  using \EQ{aKL}, we {can approximate} the
subdomain local stiffness matrix $A^{(i)}(\theta)$ as 
{{\begin{equation}\label{equation:eqA}
\begin{aligned}
  A_{st}^{(i)}(\theta)&\approx
  \int_{\Omega^{(i)}} \exp(a^{(i)}_{M_{KL}})\nabla{\phi_s}\nabla{\phi_t} d\x\\
&=\sum\limits_{{\bm\alpha}\in\mathscr{S}^{(i)}}\psi^{{(i)}}_{\bm\alpha}\int_{\Omega^{(i)}} \prod_{m=1}^{M^{(i)}_{KL}}\left\langle\exp\left(\sqrt{\lambda^{(i)}_m}a^{(i)}_m(\x)\xi^{(i)}_m\right)\psi^{{(i)}}_{\alpha_m}(\xi^{(i)}_{m})\right\rangle\nabla{\phi_s}\nabla{\phi_t} dx.
\end{aligned}
\end{equation}}}
Define {{$A^{(i)}_{\bm\alpha}$ to be the matrix with entries $A^{(i)}_{st,\bm\alpha}$ as 
$$A^{(i)}_{st,\bm\alpha}=\int_{\Omega^{(i)}} \prod_{m=1}^{M^{(i)}_{KL}}\left\langle\exp\left(\sqrt{\lambda^{(i)}_m}a^{(i)}_m(\x)\xi^{(i)}_m\right)\psi^{{(i)}}_{\bm\alpha}(\xi^{(i)}_{\bm\alpha})\right\rangle\nabla{\phi_s}\nabla{\phi_t} dx,$$}}
and we can  approximate $A^{(i)}(\theta)$ as  
\begin{align}\label{equation:Apc}
A^{(i)}(\theta)\approx A_{KL}^{(i)}(\theta)=\sum\limits_{{\bm\alpha}\in\mathscr{S}^{(i)}}A^{(i)}_{\bm\alpha}\psi^{{(i)}}_{\bm\alpha}.
\end{align}

By taking the corresponding parts from \EQ{Apc},  we obtain the
following approximations for the matrices $A^{(i)}_{rr}, A^{(i)}_{cr}$ and $A^{(i)}_{cc}$ in \EQ{preconditioned}
as
\begin{align}\label{equation:allmpc}
A^{(i)}_{rr}\approx A^{(i)}_{rr, KL},\quad  A^{(i)}_{cr}\approx
  A^{(i)}_{cr, KL}, \quad A^{(i)}_{cc}\approx A^{(i)}_{cc, KL},
\end{align}
with
\begin{align*}
A^{(i)}_{rr, KL}=\sum\limits_{{\bm\alpha}\in\mathscr{S}^{(i)}}A^{(i)}_{rr,{\bm\alpha}}\psi^{(i)}_{\bm\alpha},\quad
  A^{(i)}_{cr, KL}=\sum\limits_{{\bm\alpha}\in\mathscr{S}^{(i)}}A^{(i)}_{cr,{\bm\alpha}}\psi^{(i)}_{\bm\alpha},
  \quad
  A^{(i)}_{cc, KL}=\sum\limits_{{\bm\alpha}\in\mathscr{S}^{(i)}}A^{(i)}_{cc,{\bm\alpha}}\psi^{(i)}_{\bm\alpha}.
\end{align*}
Here $A^{(i)}_{rr,{\bm\alpha}}=\begin{bmatrix}
{A^{(i)}_{\bm\alpha}}_{II}  & {{A^{(i)^T}_{\bm\alpha}}_{\Delta I}}\\
{A^{(i)}_{\bm\alpha}}_{\Delta I}   &  {A^{(i)}_{\bm\alpha}}_{\Delta \Delta}  \\
\end{bmatrix}$, $A^{(i)}_{cr,{\bm\alpha}}=
\begin{bmatrix}
{A^{(i)}_{\bm\alpha}}_{\Pi I}\quad {A^{(i)}_{\bm\alpha}}_{\Pi\Delta}
\end{bmatrix}$,
and $A^{(i)}_{cc,{\bm\alpha}}=
{A^{(i)}_{\bm\alpha}}_{\Pi\Pi}$.



In order to construct the approximation of $ \Stilde^{-1}_\Gamma$, defined in
\EQ{StInverse}, we usually form the Cholesky factor of ${A^{(i)}_{rr}}$  in the
deterministic BDDC algorithm.  The product of 
${A^{(i)}_{rr}}^{-1}$ and a vector is calculated by  a forward and a backward
substitutions. By \EQ{phi},
we also need to calculate  $ A_{rr}^{(i)^{-1}}(\theta)
A_{cr}^{(i)^T}(\theta)$. 
When many samples are needed in the computation, it will be
more efficient to form the PC approximations   of $
A_{rr}^{(i)^{-1}}(\theta)$ and $ A_{rr}^{(i)^{-1}}(\theta)
A_{cr}^{(i)^T}(\theta)$ directly. 

In order to do that, we first discuss how we  form the PC expansion of
$Y=A_{rr}^{(i)^{-1}}(\theta)v(\theta)$, where {$v(\theta)$} is a given
vector. We assume that $Y$ and $v(\theta)$  have the 
PC expansions
as follows 
\begin{equation}\label{equation:pcvY}
  Y=\sum\limits_{{\bm\alpha}\in \mathscr{S}^{(i)}}
Y_{\bm\alpha}\psi^{(i)}_{\balpha},\quad  v(\theta)= \sum\limits_{{\bm\alpha}\in \mathscr{S}^{(i)}}
v_{\bm\alpha}\psi^{(i)}_{\balpha},
\end{equation}
where 
$Y_{\bm \alpha}=\langle Y
\psi^{(i)}_{\bm\alpha}\rangle$, which  we need to compute. 

We  rewrite $Y=A_{rr}^{(i)^{-1}}(\theta)v(\theta)$ as $
A_{rr}^{(i)}(\theta)Y=v(\theta)$ and replace $A_{rr}^{(i)}(\theta)$ by
its approximation $A_{rr,KL}^{(i)}(\theta)$ defined in \EQ{allmpc}.  We
have 
\begin{equation}\label{equation:Yeq}
  A_{rr,KL}^{(i)}(\theta)Y=v(\theta).
  \end{equation}
Let $\bm\beta(l) \in \mathscr{S}^{(i)}_d$ for $
l=1,2,\cdots,n^{(i)}_{\xi}$.  
Multiplying $\psi^{(i)}_{\bm\beta(l)}$ on both sides of
\EQ{Yeq},  taking the expectation, and using \EQ{allmpc} and
\EQ{pcvY}, we have 
\begin{align*}
\sum\limits_{{\bm\alpha}\in\mathscr{S}^{(i)}_{2d}}\sum\limits_{k=1}^{n^{(i)}_{\xi}}A^{(i)}_{rr,{\bm\alpha}}Y_{\bm\beta(k)}
\langle\psi^{(i)}_{\bm\alpha}\psi^{(i)}_{\bm\beta(k)}\psi^{(i)}_{\bm\beta{(l)}}\rangle=\sum\limits_{{\bm\alpha}\in\mathscr{S}^{(i)}_{d}} v_{\bm\alpha}\langle\psi^{(i)}_{\bm\alpha}\psi^{(i)}_{\bm\beta(l)}\rangle.
\end{align*}
Here
we can truncate the PC expansion of $A^{(i)}_{rr}$ to
${{\bm\alpha}\in\mathscr{S}^{(i)}_{2d}}$  and $v(\theta)$ to $
{{\bm\alpha}\in\mathscr{S}^{(i)}_{d}}$ without loss any accuracy and obtain
\begin{equation}\label{equation:eqIs}
A^{(i)}_{rrs}Y_{rrs}=v_{rrs},
\end{equation}
where
\begin{align}
\label{equation:Arrs}
A^{(i)}_{rrs}&=\sum\limits_{\balpha\in\mathscr{S}^{(i)}_{2d}}
\begin{bmatrix}
\langle\psi^{(i)}_{\bm\alpha}\psi^{(i)}_{\bm\beta{(1)}}\psi^{(i)}_{{\bm\beta{(1)}}}\rangle
&
\langle\psi^{(i)}_{\bm\alpha}\psi^{(i)}_{\bm\beta{(2)}}\psi^{(i)}_{{\bm\beta{(1)}}}\rangle& \cdots\langle\psi^{(i)}_{\bm\alpha}\psi^{(i)}_{\tiny{{\bm\beta}{{(n^{(i)}_{\xi})}}}}\psi^{(i)}_{{\bm\beta{(1)}}}\rangle\\
\vdots&\ddots&\vdots\\
\langle\psi^{(i)}_{\bm\alpha}\psi^{(i)}_{\bm\beta{(1)}}\psi^{(i)}_{{\bm\beta({n^{(i)}_{\xi})}}}\rangle&\langle\psi^{(i)}_{\bm\alpha}\psi^{(i)}_{\bm\beta{(2)}}\psi^{(i)}_{{\bm\beta{(n^{(i)}_{\xi})}}}\rangle&\cdots\langle\psi^{(i)}_{\bm\alpha}\psi^{(i)}_{\bm\beta{(n^{(i)}_{\xi})}}\psi^{(i)}_{{\bm\beta{(n^{(i)}_{\xi})}}}\rangle
\end{bmatrix}
\otimes A^{(i)}_{rr,{\bm\alpha}},\\
\label{equation:Yrrs}
Y_{rrs}&=\begin{bmatrix}
Y_{\bm\beta(1)}\\
Y_{\bm\beta(2)}\\
\vdots\\
Y_{\bm\beta(n^{(i)}_{\xi})}\\
\end{bmatrix},
\quad
 v_{rrs}=
  \sum\limits_{k=1}^{n^{(i)}_{\xi}}
\begin{bmatrix}
v_{\bm\beta{(k)}}\langle\psi_{\bm\beta(k)}\psi_{\bm\beta(1)}\rangle\\
 v_{\bm\beta{(k)}}\langle\psi_{\bm\beta(k)}\psi_{\bm\beta(2)}\rangle\\
\vdots\\
 v_{\bm\beta{(k)}}\langle\psi_{\bm\beta(k)}\psi_{\bm\beta(n^{(i)}_{\xi})}\rangle\\
\end{bmatrix}=\sum\limits_{k=1}^{n^{(i)}_{\xi}}e_k\otimes v_{\bm\beta{(k)}}.
\end{align}
Here 
$e_k$ is a vector  with a size  $n^{(i)}_{\xi}$  and all zero components except the $k$th component  $1$.

Solving \EQ{eqIs}, we can obtain $Y_{rrs}$ and the truncated PC coefficients of $Y$
\begin{equation}\label{equation:tpcY}
  Y\approx Y_{PC_d}=\sum\limits_{{\bm\alpha}\in \mathscr{S}^{(i)}_d}
Y_{\bm\alpha}\psi^{(i)}_{\balpha}=\sum\limits_{k=1}^{n^{(i)}_{\xi}}Y_{\bm\beta(k)}\psi ^{(i)}_{\bm\beta(k)}.
  \end{equation}

We now consider how to form the PC approximation for
{$X=A^{(i)^{-1}}_{rr}A^{(i)^T}_{cr}$}, which is needed in \EQ{phi} and
\EQ{Spi0}. We denote that the number of columns
of {$A_{cr}^{(i)^T}(\theta)$} by $n^{(i)}_c$, which is a small number of the primal
constraints chosen in Subdomain $\Omega^{(i)}$.   We can rewrite
\begin{equation}\label{equation:eqX}
  A^{(i)}_{rr}X=A^{(i)^T}_{cr}
  \end{equation}
and obtain the truncated PC approximation of $X$ column-wise for
${{\bm\alpha}\in \mathscr{S}^{(i)}_d}$
by taking the corresponding column
of $A^{(i)^T}_{cr}$ in \EQ{eqX}.  By \EQ{allmpc}, we have the PC
expansion of $A_{cr}^{(i)^T}$. Solving \EQ{eqIs} $n^{(i)}_c$ times
with $v(\theta)$ equals to the different column of $A^{(i)^T}_{cr}$
and obtain the truncated PC approximation of
$A^{(i)^{-1}}_{rr}A^{(i)^T}_{cr}$. Similarly, we can obtain the
truncated  PC approximation of $A^{(i)^{-1}}_{rr}${as $A^{(i)^{-1}}_{rr,PC_d}$} by setting
$v(\theta)$ equal to the columns of the $n^{(i)}_r\times n^{(i)}_r $
identity matrix $I^{(i)}_{rr}$, where
$n^{(i)}_r$ is the number of columns of $A^{(i)}_{cr}$.
We note that 
the columns of $I^{(i)}_{rr}$ are  constant vectors and therefore the
corresponding $v_{rrs}$ is $e_1\otimes g$, where $g$ denotes a column
of $I^{(i)}_{rr}$.

Next, we consider how to form the truncated  PC approximation  of
$Z= A^{(i)}_{cr}{A^{(i)}_{rr}}^{-1}A^{(i)}_{rc}$, which is needed in
\EQ{Spi0}. Let $Y={A^{(i)}_{rr}}^{-1}A^{(i)}_{rc}$ and define
 the truncated PC approximations of $Z$ and $Y$ for ${{\bm\alpha}\in\mathscr{S}^{(i)}_{d}}$  as
\begin{equation}\label{equation:apczy}
Z\approx Z_{PC_d}=\sum\limits_{k=1}^{n^{(i)}_{\xi}}Z_{\bm\beta(k)}\psi^{(i)}_{\bm\beta(k)},
\quad
Y\approx Y_{PC_d}=\sum\limits_{k=1}^{n^{(i)}_{\xi}}Y_{\bm\beta(k)}\psi^{(i)}_{\bm\beta(k)}.
\end{equation}
We have
\begin{equation}\label{equation:Zpc}
Z=A^{(i)}_{cr}{A^{(i)}_{rr}}^{-1}A^{(i)}_{rc}=A^{(i)}_{cr}Y.
\end{equation}
Taking expectation with $\psi^{(i)}_{\bm\beta(l)}$ for  $l=1,\cdots,n_{\xi}^{(i)}$ both sides of \EQ{Zpc} and using 
\EQ{allmpc} and \EQ{apczy}
leads to 
\begin{equation*}
\begin{bmatrix}
 Z_{\bm\beta{(1)}}\\
Z_{\bm\beta{(2)}}\\
\vdots\\
Z_{\bm\beta{(n^{(i)}_{\xi})}}
\end{bmatrix}=A^{(i)}_{crs}Y_{rrs},
\end{equation*}
where  $Y_{rrs}$ is in \EQ{Yrrs} and
\begin{equation}
\begin{aligned}
A^{(i)}_{crs}&=\sum\limits_{\alpha\in\mathscr{S}^{(i)}_{2d}}
\begin{bmatrix}
\langle\psi^{(i)}_{\bm\alpha}\psi^{(i)}_{\bm\beta{(1)}}\psi^{(i)}_{{\bm\beta{(1)}}}\rangle
&
\langle\psi^{(i)}_{\bm\alpha}\psi^{(i)}_{\bm\beta{(2)}}\psi^{(i)}_{{\bm\beta{(1)}}}\rangle& \cdots\langle\psi^{(i)}_{\bm\alpha}\psi^{(i)}_{\tiny{{\bm\beta}{{(n^{(i)}_{\xi})}}}}\psi^{(i)}_{{\bm\beta{(1)}}}\rangle\\
\vdots&\ddots&\vdots\\
\langle\psi^{(i)}_{\bm\alpha}\psi^{(i)}_{\bm\beta{(1)}}\psi^{(i)}_{{\bm\beta({n^{(i)}_{\xi})}}}\rangle&\langle\psi^{(i)}_{\bm\alpha}\psi^{(i)}_{\bm\beta{(2)}}\psi^{(i)}_{{\bm\beta{(n^{(i)}_{\xi})}}}\rangle&\cdots\langle\psi^{(i)}_{\bm\alpha}\psi^{(i)}_{\bm\beta{(n^{(i)}_{\xi})}}\psi^{(i)}_{{\bm\beta{(n^{(i)}_{\xi})}}}\rangle
\end{bmatrix}
\otimes A^{(i)}_{cr,\bm\alpha}.
\end{aligned}
\end{equation}
We note that we only need to truncate the PC approximation for
$A_{cr}^{(i)}$ in \EQ{allmpc} to
${{\bm\alpha}\in\mathscr{S}^{(i)}_{2d}}$ without loss any accuracy.

Finally, 
let  \begin{equation}\label{equation:localSpi}
  S_\Pi^{(i)}=A_{cc}^{(i)}-{A_{cr}^{(i)}A^{(i)^{-1}}_{rr}A_{cr}^{(i)^T}},\end{equation}
 the subdomain local
contribution to the coarse matrix $S_\Pi$ defined in \EQ{Spi0}.  We can obtain the PC approximation  of $S_\Pi^{(i)}$ as $S^{(i)}_{\Pi,PC_d}$, {where we 
 combine the PC coefficients of $A^{(i)}_{cc}$, defined in
 \EQ{allmpc}, and the PC coefficients of $A_{cr}^{(i)}A^{{(i)}^{-1}}_{rr}A_{cr}^{{(i)}^T}$.}

\begin{remark}
In the computation described above,  the terms we use in  the truncated PC
approximations for the subdomain matrices in \EQ{allmpc} can be
different.  For $A_{cc}^{(i)}$, we only need the PC approximation to
${{\bm\alpha}\in\mathscr{S}^{(i)}_{d}}$. But we need the truncations
of the 
matrices $A_{cr}^{(i)}$ and  $A_{rr}^{(i)}$ to
${{\bm\alpha}\in\mathscr{S}^{(i)}_{2d}}$. 
\end{remark}

\begin{remark}\label{remark:Cspd}
  {By our SG construction of $S^{(i)}_{\Pi}$, we cannot ensure the
  resulting global $S_\Pi$ is positive definite.  The difference
  between our stochastic approximate $S_{\Pi,PC_d}$ and the exact
  coarse component $S_{\Pi}$,
  which is positive definite, depends on the stochastic dimensions
  (the number of  $KL$ terms  in \EQ{aKL}) and the degree of the PC
  approximation $d$ in \EQ{sd} for each component. 
  In our
  numerical experiments, most $S_{\Pi,PC_d}$ are positive definite except one
  single case. For this case, when we increase $d$, it becomes
  positive definite.  We also plot the Frobenius norm of the difference between $S_{\Pi}$ and $S_{\Pi,PC_d}$ in
Figure \ref{fig:ScErr} with the changes of the local $KL$ terms and
$d$ for this setup.}
 \end{remark}


By the similar process of obtaining  the PC approximation of 
{$S_\Pi^{(i)}=A^{(i)}_{cc}-A_{cr}^{(i)}A^{(i)^{-1}}_{rr}A_{cr}^{(i)^T}$}, we can
obtain the PC approximation of {$S_\Gamma^{(i)} =A^{(i)}_{\Gamma\Gamma}-
A^{(i)}_{\Gamma I}{A^{(i)}_{II}}^{-1}A^{(i)^T}_{\Gamma I}$}, which is the subdomain
local Schur complement. We can use this to obtain the
approximation of the global Schur complement $S_\Gamma$ in
\EQ{preconditioned}. When we use this approximated $S_\Gamma$, we did
not solve the exact subdomain interface problem. The approximation
error has been studied in \cite{Contreras2018B}.

\subsubsection{SC}
We can also use a stochastic collocation method to obtain the PC
approximation  of the subdomain matrices $A_{cr}^{(i)}(\theta)$,
$A_{rr}^{(i)}(\theta)$ and $A_{cc}^{(i)}(\theta)$.  Recall that the PC
coefficient $A^{(i)}_{L,\balpha}$, defined in \EQ{pc}, is equal to $\langle
A^{(i)}_{L}\psi^{(i)}_{\bm\alpha}\rangle$. We can use a $M^{(i)}_{KL}$
dimensional quadrature formula to approximate the expectation $\langle
A^{(i)}_{L}\psi^{(i)}_{\bm\alpha}\rangle$ as
\begin{align}\label{equation:quadA}
A^{(i)}_{L,\bm\alpha}\approx\sum\limits_{q=1}^{Q^{(i)}}A^{(i)}_{L,\bm\alpha,q}\psi^{(i)}_{\bm\alpha}(\xi^{(i)}_q)w^{(i)}_q,
\end{align}
where $Q^{(i)}$ is the number of the quadrature points, $w^{(i)}_q$ is the weights, and 
$A^{(i)}_{L,\bm\alpha,q}$ is the realization matrix with a realization
of $\exp(a^{(i)}(\x,\theta))$ using \EQ{aKL} with $\xi^{(i)}_q$.
We can obtain the truncated PC approximations for $A_{cc}^{(i)}$,
$A_{cr}^{(i)}$, and $A_{rr}^{(i)}$ for
${\bm\alpha\in{\mathscr{S}^{(i)}_d}}$ as
\begin{align}\label{equation:allmpc_sc}
{A^{(i)}_{rr,PC_d}=}
 \sum\limits_{{\bm\alpha}\in\mathscr{S}^{(i)}_d}A^{(i)}_{rr,{\bm\alpha}}\psi^{(i)}_{\bm\alpha},\quad 
{A^{(i)}_{cr,PC_d}=}  \sum\limits_{{\bm\alpha}\in\mathscr{S}^{(i)}_d}A^{(i)}_{cr,{\bm\alpha}}\psi^{(i)}_{\bm\alpha}, \quad {A^{(i)}_{cc,PC_d}=} \sum\limits_{{\bm\alpha}\in\mathscr{S}^{(i)}_d}A^{(i)}_{cc,{\bm\alpha}}\psi^{(i)}_{\bm\alpha}.
\end{align}
In our BDDC preconditioner, defined in \EQ{preconditioned},
we need  $A^{(i)^{-1}}_{rr}$. We can use \EQ{quadA} to form the PC
approximation of the Cholesky factor of $A^{(i)}_{rr}$ directly.  We have
$A^{(i)}_{rr,\balpha,q}=R^{(i)^T}_{{rr,\balpha,q}} R^{(i)}_{{rr,\balpha,q}}$ and use \EQ{quadA} to
obtain the truncated PC approximation of the Cholesky factors of  $A^{(i)}_{rr}$ as
{\begin{align}\label{equation:R_rr}
{R^{(i)}_{rr}}\approx{R^{(i)}_{rr,PC_d}=}
  \sum\limits_{{\bm\alpha}\in\mathscr{S}^{(i)}_d} \left(
                  \sum\limits_{q=1}^{Q^{(i)}}
                  R^{(i)}_{rr,\balpha,q}\psi^{(i)}_{\bm\alpha}(\xi^{(i)}_q)w^{(i)}_q\right) \psi^{(i)}_{\bm\alpha}
                  .
\end{align}}
To avoid the 
possible non-uniqueness of the Cholesky factor,  for each realization
$A^{(i)}_{rr,\balpha,q}$, which is symmetric positive definite,  we
make sure that the diagonal elements of the Cholesky factor
$R_{{rr,\balpha,q}} $ are
positive. 

In \EQ{Spi0}, we need to form the global coarse matrix
$S_{\Pi}(\theta)$.  If the approximations of the subdomain local matrices 
$A_{cc}^{(i)}$, $A_{cr}^{(i)}$, and  $R^{(i)}_{rr}$ are not accurate
enough, the positive definiteness of the approximated
$S_{\Pi}(\theta)$ cannot be guaranteed. Following \cite[Section
3.4.2]{stochasticpRMCM2021},
for each realization
$S_{\Pi, \balpha,q}^{(i)}$, the subdomain local contribution defined
in \EQ{localSpi}, we 
do
the following eigen-decomposition $S^{(i)}_{\Pi,\balpha,q}=QDQ^T$. Here
$D$ is the positive diagonal  matrix with the eigenvalues of
$S^{(i)}_{\Pi,\balpha,q}$ and $Q$ is the orthogonal matrix with the
eigenvectors as the columns.  We define
$H^{(i)}_{rr,\balpha,q}=QD^{\frac{1}{2}}Q^T$ and use \EQ{quadA} to
obtain the truncated PC approximation of $H^{(i)}_\Pi$ as
{
\begin{align}\label{equation:H_sc}
{H^{(i)}_{\Pi}}\approx{{H^{(i)}_{\Pi,PC_d}}=}
  \sum\limits_{{\bm\alpha}\in\mathscr{S}^{(i)}_d} \left(
                  \sum\limits_{q=1}^{Q^{(i)}}
                  H^{(i)}_{\Pi,\balpha,q}\psi^{(i)}_{\bm\alpha}(\xi^{(i)}_q)w^{(i)}_q\right) \psi^{(i)}_{\bm\alpha}
\end{align}
}
The subdomain matrix $S_\Pi^{(i)}\approx
H^{(i)}_{\Pi}H^{(i)^T}_{\Pi}$, {which is positive definite.}

Similarly, we can also obtain the truncated PC approximation for the
subdomain local Schur complement {$S^{(i)}_\Gamma$}. 

\begin{algorithm}
\caption{The Stochastic BDDC Algorithm 
}\label{alg:cap}
\begin{algorithmic}

  \State{{{Offline}}:}
  \State{Generate the PC coefficients: $A^{(i)}_{cc,\bm\alpha}$,
    $A^{(i)}_{cr,\bm\alpha}$, $A^{(i)}_{\Gamma\Gamma,\bm\alpha}$ ( inexact $S_\Gamma$),
    $A^{(i)}_{\Gamma I,\bm\alpha}$ (inexact $S_\Gamma$)
    \begin{itemize}
      \item SG: 
    $A^{(i)^{-1}}_{rr,\bm\alpha}$,
    $\left(A^{(i)^{-1}}_{rr}A^{(i)}_{cr}\right)_{\bm\alpha}$, 
    $S_{\Pi,,\bm\alpha}^{(i)}$,  and 
    $\left(A^{(i)^{-1}}_{II}A^{(i)}_{I\Gamma}\right)_{\bm\alpha}$
    (inexact $S_\Gamma$)
    \item SC: 
$R^{(i)}_{rr,\bm\alpha}$,  
$H_{\Pi,\bm\alpha}^{(i)}$, and $R^{(i)}_{II,\bm\alpha}$ (inexact $S_\Gamma$) 
   
\end{itemize}
}
\State{Online:}
\For{each sample $\theta$}
\For{each subdomain $\Omega^{(i)}_i$}
\State{Get the local Hermite basis in each subdomain
  $\psi_{\bm\alpha}(\xi^{(i)}_{\bm\alpha})$ with
  $\xi^{(i)}_{\bm\alpha}$ defined as \EQ{thetaprojection}}
\State{Generate stiffness matrix in each subdomain ${A}^{(i)}$ using
  finite element calculation or the PC approximation} 

\State{Generate the {truncated} PC expansion matrices 
     \begin{itemize}
      \item SG: 
    {$A^{(i)^{-1}}_{rr,PC_d}$,
    $(A^{(i)^{-1}}_{rr}A^{(i)}_{cr})_{PC_d}$, and
    $S_{\Pi,PC_d}^{(i)}$;
    \item SC: 
$R^{(i)}_{rr,PC_d}$,  $A^{(i)}_{cr,PC_d}$, and
$H_{\Pi,PC_d}^{(i)}$}.
\end{itemize}
}
\EndFor
\State{Generate the stochastic  coarse problem defined in \EQ{Spi0}
  using the PC approximations of $S_\Pi^{(i)}$.}
\State{Use $CG$
  with the stochastic BDDC preconditioner
  to solve \EQ{preconditioned} }
\EndFor
\end{algorithmic}
\end{algorithm}

Our stochastic BDDC algorithms have be summarized in Algorithm \ref{alg:cap}.

\section{{Analysis of the stochastic BDDC algorithm}}

Given a sample $\theta$, recall the partial assembled matrix
$\widetilde{A}(\theta)$ is defined in \EQ{anothertwoparts}. In this
section, most of our matrices are for a given $\theta$. To make our
notation simpler, we do not write $\theta$ explicitly.  The partially
assembled Schur complement $\Stilde_\Gamma$ is defined in \EQ{schurT}
and $\Stilde^{-1}_\Gamma$ in the BDDC preconditioner is defined in
\EQ{StInverse}.

In previous section, we discuss how we can construct our stochastic approximations of
$\Stilde_\Gamma$ and $\Stilde^{-1}_\Gamma$, denoted by $\Stilde_{\Gamma,F}$ and $\Stilde^{-1}_{\Gamma,P}$,
respectively. The stochastic
approximation of the global Schur complement $S_\Gamma$ is denoted as
$S_{\Gamma,F}=\Rtilde^T\Stilde_{\Gamma,F}\Rtilde$. {We provide the condition number estimates of the stochastic BDDC algorithms constructed by the SG method.   A similar approach can be applied for those constructed by the SC method as well.
Let $\widetilde{A}_{\Gamma\Gamma,F}=\widetilde{A}_{\Gamma\Gamma,PC_d}$, and 
{$\widetilde{A}_{\Gamma I,F}A^{-1}_{II,F}\widetilde{A}^T_{\Gamma I,F}
=(\widetilde{A}_{\Gamma I}A^{-1}_{II}\widetilde{A}^T_{\Gamma I})_{PC_d},$ where $(\widetilde{A}_{\Gamma I}A^{-1}_{II}\widetilde{A}^T_{\Gamma I})_{PC_d}$ is generated similarly as $Z_{PC_d}$ in \EQ{apczy} by replacing $Z$ as
$\widetilde{A}_{\Gamma I}A^{-1}_{II}\widetilde{A}^T_{\Gamma I}$ and $Y$ as $A^{-1}_{II}\widetilde{A}^T_{\Gamma I}$, respectively. }}Let  $A_{rr,P}^{{(i)}^{-1}}=A_{rr,PC_d}^{{(i)}^{-1}}$,
$ \Phi_P= R_{\Gamma \Pi}^T - R^T_{\Gamma \Delta}\sum_{i=1}^{N}
  \left[ \vvec{0} ~ R^{(i)^T}_{\Delta}
 \right](A_{rr}^{{(i)}^{-1}}A_{cr}^{(i)^T})_{PC_d}
R^{(i)}_{\Pi},$ $S_{\Pi,P}^{-1}=S_{\Pi,PC_d}^{-1},$
 and $S_{\Pi,PC_d}=\sum\limits_{i=1}^NR^{(i)^T}_{\Pi}S^{(i)}_{\Pi,PC_d}R^{(i)}_{\Pi}.$ 
We have
\begin{equation}\label{equation:StF}
  \Stilde_{\Gamma,F}=\widetilde{A}_{\Gamma\Gamma,F}-
{\widetilde{A}_{\Gamma I,F}A^{-1}_{II,F}\widetilde{A}^{T}_{\Gamma I,F}}
  \end{equation}
  and
  \begin{equation}\label{equation:StP}
  \Stilde^{-1}_{\Gamma,P}=R_{\Gamma\Delta}^T\left(\sum_{i=1}^{N}
 \left[ \vvec{0} ~ R^{(i)^T}_{\Delta}
 \right] {A_{rr,P}^{{(i)}^{-1}}} \left[ \begin{array}{c} \vvec{0}  \\
R^{(i)}_{\Delta} \end{array} \right] \right)R_{\Gamma \Delta} + \Phi_P
S_{\Pi,P}^{-1} \Phi_P^T. 
  \end{equation}

{ In the rest of the section, we will analyze the
  condition number of our BDDC preconditioned system. We first
  approximate $\kappa$ in \EQ{pde1} using a global KL
  approximation $\exp(a_{M_{KL}})$, defined in
  \EQ{globalKL}.  The 
  local matrices $A_L^{(i)}$,  the different Schur complements
  $S_\Gamma$, $\Stilde_\Gamma$, and
  $\Stilde^{-1}_\Gamma$ are constructed using $a_{M_{KL}}$.  The
  stochastic approximations $\Stilde^{-1}_{\Gamma,P}$ and
  $\Stilde_{\Gamma,F}$ are constructed using the local KL
  approximation {$a_{M_{KL}}^{(i)}$}, defined in \EQ{aKL}, and their PC
  approximations, constructed in Section \ref{sec:SG}.  We will prove that, for a fixed mesh with
the parameters $H$ (the size of the subdomain) and $h$ (the size of
the mesh), 
the probability of {the event such that the BDDC preconditioned
operator $\Stilde^{-1}_{\Gamma,P}\Stilde_{\Gamma,F}$ is well
conditioned}, will go to $1$ as the global and local KL terms $M_{KL}$,
$M^{(i)}_{KL}$, defined in \EQ{globalKL} and \EQ{aKL} respectively,
and the PC degree $d$, defined in \EQ{sd},  go to infinity.  }


{Let 
$
c_0=\langle \exp(a^{(i)}(\x,\theta))\rangle$ denote the expectation of
$\exp(a^{(i)}(\x,\theta))$. Under the 
covariance matrix defined in \EQ{cov}, $c_0=e^{\frac12 \sigma^2}$ which is independent of $\x$.}}
On the other hand,  {{using the KL expansion \EQ{LML}}}, we  have
\begin{equation}\label{equation:c0}
  c_0=\langle \exp(a^{(i)}(\x,\theta))\rangle=\langle \exp(\sum\limits_{m=1}^{\infty}\sqrt{\lambda^{(i)}_m}a^{(i)}_m(\x)\xi^{{(i)}}_m)\rangle=\exp(\frac12\sum\limits_{m=1}^{\infty}\lambda^{(i)}_ma^{(i)^2}_m(\x)),
\end{equation}
therefore, {{$\sum\limits_{m=1}^{\infty}\lambda^{(i)}_ma^{(i)^2}_m(\x)$
is uniformly bounded.}} We have
\begin{equation}\label{equation:L2sum}
\sum\limits_{m=1}^{\infty}\lambda^{(i)}_m\|a^{(i)}_m\|^2_2= \sum\limits_{m=1}^{\infty}{\lambda^{(i)}_m}\int_{\Omega^{(i)}}{a^{{(i)}^2}_m(\x)}d\x=\int_{\Omega^{(i)}}\sum\limits_{m=1}^{\infty}{\lambda^{(i)}_m}{a^{{(i)}^2}_m(\x)}d\x<\infty.
\end{equation}

{Similarly, we have the result for the global KL terms
\begin{equation}\label{equation:gL2sum}
\sum\limits_{m=1}^{\infty}\lambda_m\|a_m\|^2_2< \infty.
  \end{equation}
}

We also have 
\begin{eqnarray}\label{equation:ec0}
&&{\sum\limits_{{\bm\alpha}\in {\mathscr{S}^{(i)}}}\frac{(2\ln(c_0))^{|\bm\alpha|}}{{\bm\alpha}!}
   }\nonumber\\
  &=&\sum\limits_{\alpha_1\in \mathbb{N}_0,\alpha_2\in \mathbb{N}_0,\cdots,\alpha_{M_{KL}^{(i)}}\in \mathbb{N}_0}\frac{(2\ln(c_0))^{\alpha_1}}{{\alpha_1}!}\frac{(2\ln(c_0))^{\alpha_2}}{{\alpha_2}!}\cdots \frac{(2\ln(c_0))^{ \alpha_{M_{KL}^{(i)}}}}{{\alpha_{M_{KL}^{(i)}}}!}\nonumber\\
&\leq&(\sum\limits_{\alpha_1\in \mathbb{N}_0}\frac{(2\ln(c_0))^{\alpha_1}}{\alpha_1!})^{M_{KL}^{(i)}}=e^{2\ln(c_0)M_{KL}^{(i)}}=c_0^{2 M_{KL}^{(i)}}<\infty.
\end{eqnarray}

Therefore, let $d_\Omega$ be the diameter of $\Omega$, given $\epsilon>0$, {$h_0>0$, } using \EQ{c0}, \EQ{L2sum}, and
\EQ{ec0}, we can make the following assumptions: for all $h>h_0$, 

\begin{assumption}\label{assump:PC}
  
  we assume that
  
\begin{enumerate*}[label =\arabic*.,itemjoin=\\]

\item {$M_{KL}$, $M_{KL}^{(i)}$ are large enough such that the KL
    truncation errors in \EQ{globalKLerror} and \EQ{localerror}  satisfy
{$\sum\limits_{m=M_{KL}+1}^{\infty}\lambda_m\|a_m\|^2_2\leq\epsilon^2 h_0^2(\frac{h_0}{d_\Omega})^6$ and $\sum\limits_{m=M^{(i)}_{KL}+1}^{\infty}\lambda^{(i)}_m\|a^{(i)}_m\|^2_2\leq\epsilon^2 h_0^2(\frac{h_0}{d_\Omega})^6$ ;}}
\item {Given $M^{(i)}_{KL}$}, $d$ is large enough such that
{${\sum\limits_{{\bm\alpha}\in {\mathscr{S}^{(i)}\backslash \mathscr{S}^{(i)}_d }}{c^2_0}\frac{(2\ln(c_0))^{|\bm\alpha|}}{{\bm\alpha}!}
\leq\epsilon\frac{1}{d_\Omega^2}}$.}

\end{enumerate*}
\end{assumption}

\begin{lemma}\label{pcestimate}
We have 
{\begin{equation}\label{equation:LPCd}
E\|{A^{(i)}_{L}}-{A^{{(i)}}_{L,PC_d}}\|^2_2\leq C{\epsilon}.
\end{equation}}
Recall $A^{(i)}_{L}$ is any subdomain local matrices constructed using
$a_{ML}$ and $A^{{(i)}}_{L,PC_d}$ is its PC approximation, defined in
\EQ{pc} and based on {$a_{M_{KL}}^{(i)}$}. 
\end{lemma}
\begin{proof}
Let $R_1\times R_2$ be the dimension of   $A^{(i)}_{L}$.
$R_1=C\frac{H}{h}$ and $R_2=C\frac{H}{h}$.  Since
$\|{A^{(i)}_{L}}-{A^{{(i)}}_{L,PC_d}}\|^2_2\le
\|{A^{(i)}_{L}}-{A^{{(i)}}_{L,PC_d}}\|^2_F$,  we only need to estimate
$E\|{A^{(i)}_{L}}-{A^{{(i)}}_{L,PC_d}}\|^2_F$.

We denote
\begin{equation}\label{equation:defa}
\begin{aligned}
a_{M_{KL}}(\x,\theta)&=\sum\limits_{m=1}^{{M_{KL}}}\sqrt{\lambda_m}{a}_m(\x)\xi_m(\theta),\\
  a^{(i)}(\x,\theta)&=\sum\limits_{m=1}^{\infty}\sqrt{\lambda_m}{a}_m(\x)\xi_m(\theta)=\sum\limits_{m=1}^{\infty}\sqrt{\lambda^{(i)}_m}{a}^{(i)}_m(\x)\xi^{(i)}_m(\theta),\\
a_{M_{KL}}^{(i)}(\x,\theta)&=\sum\limits_{m=1}^{{M^{(i)}_{KL}}}\sqrt{\lambda^{(i)}_m}{a}^{(i)}_m(\x)\hat{\xi}^{(i)}_m(\theta),\\
\hat{a}_{M_{KL}}^{(i)}(\x,\theta)&=\sum\limits_{m=1}^{{M^{(i)}_{KL}}}\sqrt{\lambda^{(i)}_m}{a}^{(i)}_m(\x)\xi^{(i)}_m(\theta).
\end{aligned}
\end{equation}
We have
\begin{align*}
&E(\|(A^{(i)}_{L}-A^{(i)}_{L,PC_d})\|^2_F)\\
&=E\left(\sum\limits_{l=1}^{R_1}\sum\limits_{m=1}^{R_2}\left(\int_{\Omega^{(i)}}\left(e^{{a}_{M_{KL}}(\x,\theta)}-\sum\limits_{\bm\alpha\in {\S}^{(i)}_d}\langle e^{a^{(i)}_{M_{KL}}(\x,\theta)}\psi_{\bm\alpha}\rangle\psi_{\bm\alpha}\right)
\nabla\phi_l\nabla\phi_md\x\right)^2\right)\\
&\leq\sum\limits_{l=1}^{R_1}\sum\limits_{m=1}^{R_2}
E\int_{\Omega^{(i)}}\left(e^{a_{M_{KL}}(\x,\theta)}-\sum\limits_{\bm\alpha\in {\S}^{(i)}_d}\langle e^{a^{(i)}_{M_{KL}}(\x,\theta)}\psi_{\bm\alpha}\rangle\psi_{\bm\alpha}\right)^2d\x
\int_{\Omega^{(i)}}(\nabla\phi^{(i)}_l\nabla\phi^{(i)}_m)^2d\x\\
&\leq\frac{CH^2}{h^4}
                                                                   E\int_{\Omega^{(i)}}\left(e^{a_{M_{KL}}(\x,\theta)} -\sum\limits_{\bm\alpha\in {\S}^{(i)}_d}\langle e^{a^{(i)}_{M_{KL}}(\x,\theta)}\psi_{\bm\alpha}\rangle\psi_{\bm\alpha}\right)^2d\x,
\end{align*}
where the first inequality follows from the Cauchy Schwarz inequality and
the second inequality follows from the fact 
that $\phi_l$, $\phi_m$ are basis functions for the spacial  variable 
with $|\nabla \phi_l |\leq C\frac 1h$,  $R_1=C\frac{H}{h}$, and $R_2=C\frac{H}{h}$.
\begin{equation}\label{equation:estimatea}
\begin{aligned}
  &E\int_{\Omega^{(i)}}\left(e^{{a_{M_{KL}}(\x,\theta)}}-\sum\limits_{\bm\alpha\in {\S}^{(i)}_d}\langle e^{a^{(i)}_{M_{KL}}(\x,\theta)}\psi_{\bm\alpha}\rangle\psi_{\bm\alpha}\right)^2d\x\\
\le&C E\int_{\Omega^{(i)}}\left(e^{a_{M_{KL}}(\x,\theta)}-e^{a^{(i)}_{M_{KL}}(\x,\theta)}\right)^2d\x + CE\int_{\Omega^{(i)}}\left(e^{a^{(i)}_{M_{KL}}(\x,\theta)}-\sum\limits_{\bm\alpha\in {\S}^{(i)}_d}\langle e^{a^{(i)}_{M_{KL}}(\x,\theta)}\psi_{\bm\alpha}\rangle\psi_{\bm\alpha}\right)^2d\x\\
\le&C E\int_{\Omega^{(i)}}\left(e^{a_{M_{KL}}(\x,\theta)}-
                                                                                                                                                                                         e^{\hat{a}^{(i)}_{M_{KL}}(\x,\theta)}\right)^2d\x + CE\int_{\Omega^{(i)}}\left(e^{a^{(i)}_{M_{KL}}(\x,\theta)}-\sum\limits_{\bm\alpha\in {\S}^{(i)}_d}\langle e^{a^{(i)}_{M_{KL}}(\x,\theta)}\psi_{\bm\alpha}\rangle\psi_{\bm\alpha}\right)^2d\x\\
\le &CE\int_{\Omega^{(i)}}\left(e^{a_{M_{KL}}(\x,\theta)}-e^{a(\x,\theta)}\right)^2d\x + CE\int_{\Omega^{(i)}}\left(e^{a^{(i)}(\x,\theta)} -e^{\hat{a}^{(i)}_{M_{KL}}(\x,\theta)}\right)^2d\x\\
&+CE\left(\int_{\Omega^{(i)}}\left(\sum\limits_{\bm\alpha\in {\S}^{(i)}}\langle e^{a^{(i)}_{M_{KL}}(\x,\theta)}\psi_{\bm\alpha}\rangle\psi_{\bm\alpha}-\sum\limits_{\bm\alpha\in {\S}^{(i)}_d}\langle e^{a^{(i)}_{M_{KL}}(\x,\theta)}\psi_{\bm\alpha}\rangle\psi_{\bm\alpha}\right)^2\right)\\
&:=I+II+III,
\end{aligned}
\end{equation}
where the second inequality follows the definition of $a_{M_{KL}}^{(i)}$
in \EQ{defa} and   \EQ{localxi}. The last inequality uses $a(\x,\theta)=a^{(i)}(\x,\theta),$ for $x\in\Omega^{(i)}$.
{ Therefore, $\|(A^{(i)}_{L}-A^{(i)}_{L,PC_d})\|^2_F $ can be
  bounded by the sum of $I$, $II$, and $III$. $I$ is due to the global KL truncation; $II$ is due to the subdomain
local KL truncation; and $III$ is due to the PC approximation. 
{{We will estimate $II$ and $III$, and $I$ can be estimated using a similar method as $II$.}}}


For the  estimate of $II$, we have
 \begin{equation}\label{equation:II1}
\begin{aligned}
&\int_{\Omega^{(i)}}E(e^{a^{(i)}(\x,\theta)}-e^{\hat{a}^{(i)}_{M_{KL}}(\x,\theta)})^2d\x
\leq C\int_{\Omega^{(i)}}(E({{e^{{4}{\hat{a}}^{(i)}_{M_{KL}}(\x,\theta)}}}))^{\frac12}(E(e^{a^{(i)}(\x,\theta)-{\hat{a}}^{(i)}_{M_{KL}}(\x,\theta)}-1)^4)^{\frac12} d\x\\
&\leq C \left(\int_{\Omega^{(i)}}(E({{e^{{4}{\hat{a}}^{(i)}_{M_{KL}}(\x,\theta)}}}))d\x\right)^{\frac12}\left(\int_{\Omega^{(i)}}(E(e^{a^{(i)}(\x,\theta)-{\hat{a}}^{(i)}_{M_{KL}}(\x,\theta)}-1)^4)d\x\right)^{\frac12}.
\end{aligned}
\end{equation}

\begin{eqnarray}\label{equation:II11}
 &&
    \int_{\Omega^{(i)}}(E({{e^{{4}{\hat{a}}^{(i)}_{M_{KL}}(\x,\theta)}}}))d\x\\
  &=&\int_{\Omega^{(i)}}E\left(\exp({4}\sum\limits_{m=1}^{M^{(i)}_{KL}}\sqrt{\lambda^{(i)}_m}a^{(i)}_m(\x)\xi^{(i)}_m)\right)d\x
      =\int_{\Omega^{(i)}}\exp({8}\sum\limits_{m=1}^{M^{(i)}_{KL}}{\lambda^{(i)}_m}a^{(i)^2}_m(\x)))d\x\le C,\nonumber
\end{eqnarray}
where we use \EQ{c0} in the last step.


{We denote the space measure on $\Omega$ as $\mu$ and let}
\begin{equation}\label{equation:da}
  \begin{aligned}
  \Delta
  &a=a^{(i)}(\x,\theta)-{\hat{a}}^{(i)}_{M_{KL}}(\x,\theta)=\sum\limits_{m=M^{(i)}_{KL}+1}^{\infty}\sqrt{\lambda^{(i)}_m}a^{(i)}_m(\x)\xi^{(i)}_m,\\
  &E(\exp(\Delta a))=\exp(\frac12\sum\limits_{m=M^{(i)}_{KL}+1}^{\infty}\lambda^{(i)}_ma^{(i)^2}_m(\x)).
\end{aligned}
\end{equation}

\begin{equation}\label{equation:II12}
  \begin{aligned}
 & \int_{\Omega^{(i)}}(E(e^{a^{(i)}(\x,\theta)-{\hat{a}}^{(i)}_{M_{KL}}(\x,\theta)}-1)^4)d\x= 
\int_{\Omega^{(i)}}E\left(\exp\left(\Delta a\right)-1\right)^4d\x\\
&=\int_{\Omega^{(i)}}E\left(1-4\exp(\Delta a)+6\exp(2 \Delta a)
-4\exp(3 \Delta a)+\exp(4 \Delta a)\right)d\x\\
&=\int_{\Omega^{(i)}}\left(1-4\exp(\frac12\sum\limits_{m=M^{(i)}_{KL}+1}^{\infty}\lambda^{(i)}_ma^{(i)^2}_m(\x))+6\exp(2\sum\limits_{m=M^{(i)}_{KL}+1}^{\infty}{\lambda^{(i)}_m}a^{(i)^2}_m(\x))\right.\\
&\left.-4\exp(\frac92\sum\limits_{m=M^{(i)}_{KL}+1}^{\infty}\lambda^{(i)}_ma^{(i)^2}_m(\x))+\exp(8\sum\limits_{m=M^{(i)}_{KL}+1}^{\infty}{\lambda^{(i)}_m}a^{(i)^2}_m(\x))\right)d\x\\
&\leq \int_{\Omega^{(i)}} \left(6\exp(2\sum\limits_{m=M^{(i)}_{KL}+1}^{\infty}{\lambda^{(i)}_m}a^{(i)^2}_m(\x))+\exp(8\sum\limits_{m=M^{(i)}_{KL}+1}^{\infty}{\lambda^{(i)}_m}a^{(i)^2}_m(\x))\right)d\x-7H^2,
\end{aligned}
\end{equation}
where the last inequality follows from the fact that
{{$\sum\limits_{m=M^{(i)}_{KL}+1}^{\infty}\lambda^{(i)}_ma^{{(i)}^2}_m(\x)\geq0$}}
and $\mu({\Omega^{(i)}})=H^2$.

By \EQ{L2sum} and Assumption \ref{assump:PC},  we have $$
\int_{\Omega^{(i)}}\sum\limits_{m=M^{(i)}_{KL}+1}^{\infty}{\lambda^{(i)}_m}{a^{{(i)}^2}_m(\x)}d\x=\sum\limits_{m=M^{(i)}_{KL}+1}^{\infty}{\lambda^{(i)}_m}\|{a^{{(i)}^2}_m(\x)}\|^2_2\le \epsilon^2 \frac{h^{8}}{H^6}
.$$

 {Let, $A=\bigg\{\x\in
 \Omega^{(i)}:\sum\limits_{m=M^{(i)}_{KL}+1}^{\infty}{\lambda^{(i)}_m}a^{{(i)}^2}_m>
 \epsilon \frac{h^4}{H^4}\bigg\}$,
we  have $\mu(A) \leq{\epsilon \frac{h^4}{H^2}}$.}

Therefore,
\begin{equation}\label{equation:II13}
\begin{aligned}
&\int_{\Omega^{(i)}}\left(6\exp(2\sum\limits_{m=M^{(i)}_{KL}+1}^{\infty}{\lambda^{(i)}_m}a^{(i)^2}_m(\x))+\exp(8\sum\limits_{m=M^{(i)}_{KL}+1}^{\infty}{\lambda^{(i)}_m}a^{(i)^2}_m(\x))\right)d\x\\
&=\int_{\Omega^{(i)}\cap A}\left(6\exp(2\sum\limits_{m=M^{(i)}_{KL}+1}^{\infty}{\lambda^{(i)}_m}a^{(i)^2}_m(\x))+\exp(8\sum\limits_{m=M^{(i)}_{KL}+1}^{\infty}{\lambda^{(i)}_m}a^{(i)^2}_m(\x))\right)d\x\\
&\quad+\int_{\Omega^{(i)}\cap A^c}\left(6\exp(2\sum\limits_{m=M^{(i)}_{KL}+1}^{\infty}{\lambda^{(i)}_m}a^{(i)^2}_m(\x))+\exp(8\sum\limits_{m=M^{(i)}_{KL}+1}^{\infty}{\lambda^{(i)}_m}a^{(i)^2}_m(\x))\right)d\x\\
&\leq C\mu(A)+(6\exp({2\epsilon\frac{h^4}{H^4}})+\exp({8\epsilon\frac{h^4}{H^4}})\mu(\Omega^{(i)})\\
&\leq C\epsilon\frac{h^4}{H^2}+(6\exp({2\epsilon\frac{h^4}{H^4}})+\exp({8\epsilon\frac{h^4}{H^4}}))H^2,
\end{aligned}
\end{equation}
where we use \EQ{c0} for the last second inequality. 

Combining \EQ{II1},\EQ{II11}, \EQ{II12}, and \EQ{II13}, we
have  
\begin{align*}
II&\leq C\left(C\epsilon\frac{h^4}{H^2}+(6\exp({2\epsilon\frac{h^4}{H^4}})+\exp({8\epsilon\frac{h^4}{H^4}}))H^2-7H^2\right)\leq C\frac{h^4}{H^2}\epsilon.
\end{align*}

For the estimate of III, using the fact that 
$\{\psi_\alpha\}$ forms an orthonormal basis, we have
\begin{equation}\label{equation:III}
\begin{aligned}
  III&=E\left(\int_{\Omega^{(i)}}\sum\limits_{\bm\alpha\in
       \S^{(i)}\backslash \S^{(i)}_d}\langle
       e^{a^{(i)}_{M_{KL}}(\x,\theta)}\psi_{\bm\alpha}\rangle\psi_{\bm\alpha}d\x\right)^2
  =\int_{\Omega^{(i)}}\sum\limits_{\bm\alpha\in \S^{(i)}\backslash \S^{(i)}_d}\langle e^{a^{(i)}_{M_{KL}}(\x,\theta)}\psi_{\bm\alpha}\rangle^2d\x.
\end{aligned}
\end{equation}
By \cite[Page 926]{Ullmann2010} or \cite[Chapter I, Theorem
3.1]{malliavin2015stochastic}, we have 
 \begin{equation}\label{equation:errorterm}
 \begin{aligned}
 &\sum\limits_{\bm\alpha\in \S^{(i)}\backslash \S^{(i)}_d}\langle\exp({a^{(i)}_{M_{KL}}(\x,\theta)})\psi_{\bm\alpha}\rangle^2\\
&{=\sum\limits_{{\bm\alpha}\in {{\S}^{(i)}\backslash\mathscr{S}^{(i)}_d }}
{\langle \exp(a^{(i)}_{M_{KL}}(\x,\theta))\rangle^2}
\frac{1}{{\bm\alpha}!}\left(\prod\limits_{m=1}^{M^{(i)}_{KL}}(\sqrt{\lambda^{(i)}_m}a^{(i)}_m(\x))^{2\alpha^{(i)}_m}\right)}\\
&\leq\sum\limits_{{\bm\alpha}\in {{\S}^{(i)}\backslash\mathscr{S}^{(i)}_d }}{c^2_0}
{\frac{1}{{\bm\alpha}!}\left(\left(\sum\limits_{m=1}^{M_{KL}^{(i)}}{\lambda^{(i)}_m}(a^{(i)}_m(\x))^2\right)^{\sum\limits_{m=1}^{M^{(i)}_{KL}}\alpha^{(i)}_m}\right)}\\
&{\leq \sum\limits_{{\bm\alpha}\in {{\S}^{(i)}\backslash\mathscr{S}^{(i)}_d }}c^2_0\frac {(2\ln(c_0))^{{|\bm\alpha|}} }{{\bm\alpha}!}
\le \epsilon\frac{1}{H^2},}
\end{aligned}
\end{equation}
where we use \EQ{c0} for the first inequality and Assumption
\ref{assump:PC} in the last step.

Plugging the above estimate  into \EQ{III}, we
obtain $III\le C\epsilon$. 

 \end{proof}

 \begin{lemma}\label{pcinverseestimate}
   There exist positive constants $N$, $N^{(i)}$, and $M$ such that if
   $M_{KL}>N$, $M_{KL}^{(i)}>N^{(i)}$, and $d>M$, we have
\begin{align*}
E\|A^{{(i)}^{-1}}_{L}-A^{{(i)}^{-1}}_{L,PC_d}\|^2_2\leq C\epsilon.
\end{align*}
\begin{proof}
 Let $e_j$, $v_j$ and $w_j$ be the $j$th column of the identity
 matrix, $A^{{(i)}^{-1}}_{L}$ and
 $A^{{(i)}^{-1}}_{L,PC_d}$, respectively. { Let $R_1$ be  the number of
 column of $A^{{(i)}^{-1}}_{L}$ with the size $\frac{H}{h}$. }
 We have
 $${E\|A^{{(i)}^{-1}}_{L}-A^{{(i)}^{-1}}_{L,PC_d}\|^2_2\le
 E\|A^{{(i)}^{-1}}_{L}-A^{{(i)}^{-1}}_{L,PC_d}\|^2_F =  E(\sum\limits_{j=1}^{R_1}\|v_j-w_j\|^2_{l^2}).}$$
 By the definition, we have
 $A^{{(i)}}_{L}v_j=e_j$ and $v_j$ is the finite element solution of
 \EQ{pde1} with  $\kappa(\x,\theta)=e^{a_{M_{KL}}}$ on $\Omega^{(i)}$ and
 corresponding boundary conditions on $\partial\Omega^{(i)}$. The right
 hand side is corresponding to $e_j$. Similar to the proof of Lemma
 \ref{pcestimate}, we introduce two additional finite element solutions of \EQ{pde1}
 on $\Omega^{(i)}$ with the same boundary condition and the right hand
 side function but different coefficients
 $\kappa(\x,\theta)$. $\hat{v}_j$ and $\hat{w}_j$ are the solutions with
 $\kappa(\x,\theta)=e^{a^{(i)}(\x,\theta)}$  and
 $\kappa(\x,\theta)=e^{a^{(i)}_{M_{KL}}}$,
 respectively. {Using the estimate in Lemma
   \ref{pcestimate} and following the
 convergence  analysis for the solution with truncated KL expansion of
 the coefficients and the finite element error analysis in
 \cite{CPDE2012,CST2013}, \cite[Theorem 8]{GKNSSS2015},  there exist two positive
 constants $N$ and $N^{(i)}$ such that if $M_{KL}>N$ and
 $M^{(i)}_{KL}>N^{(i)}$, we have $E\|v_j-\hat{v}_j\|_{l^2}^2\le
 \frac{\epsilon}{3}{\frac hH}$ and $E\|\hat{w}_j-\hat{v}_j\|_{l^2}^2\le
 \frac{\epsilon}{3}{\frac hH}$. } By \EQ{tpcY}, we know $w_j$ is the stochastic
 finite element solution of \EQ{pde1} with
 $\kappa(\x,\theta)=e^{a^{(i)}_{ML}}$ and the same boundary conditions
 and the right hand side functions as for $\hat{w}_j$.  Using the
 convergence of the stochastic finite element solutions
 \cite{BNTcollocation2007,GS2009,Gittelson2010,MS2013,HS2014,DS2017,BCDM2017}, with fixed $M^{(i)}_{KL}$, there exists a positive $M>0$ such that $E\|w_j-\hat{w}_j\|_{l^2}^2\le
 \frac{\epsilon}{3}{\frac hH}$ if
 $d>M$. 
 Combing all these estimates, we have $E\|v_j-w_j\|_{l^2}^2\le \epsilon\frac hH$. 
 
 Therefore, we have
 \begin{align*}
 E(\sum\limits_{j=1}^{R_1}\|v_j-w_j\|^2_{{l_2}})\leq R_1\epsilon\frac hH\leq C\epsilon,
 \end{align*}
 which implies the result.

\end{proof}
\end{lemma}

\begin{lemma}\label{lemma:SF}
For $u_\Gamma\in\widetilde{W}_\Gamma$ and $\theta\in\Theta$, we have
\begin{align*}
|u^T_\Gamma\widetilde{S}_{\Gamma} u_\Gamma-u^T_\Gamma\widetilde{S}_{\Gamma,F} u_\Gamma|\leq \gamma_{FS}\|u_{\Gamma}\|^2_{\widetilde{S}_{\Gamma}},
\end{align*}
where
\begin{equation}\label{equation:gammaFS}
  \begin{aligned}
  \gamma_{FS}&=\left(\|\Atilde_{\Gamma\Gamma}-\Atilde_{\Gamma\Gamma,F}\|_2
+\|\Atilde_{\Gamma I}-\Atilde_{\Gamma I,F}\|_2\|A^{-1}_{II}{\Atilde^T_{\Gamma I}}\|_2\right.\\
&+\left.\|\Atilde_{\Gamma
    I,F}\|_2\left(\|A^{-1}_{II}-A^{-1}_{II,F}\|_2\|{\Atilde^T_{\Gamma
        I}}\|_2+\|A^{-1}_{II,F}\|_2\|{\Atilde^T_{\Gamma
        I}}-{\Atilde^T_{\Gamma I,F}}\|_2\right)\right)
\|\widetilde{S}^{-1}_\Gamma\|_2.
\end{aligned}
\end{equation}
\end{lemma}
\begin{proof}
\begin{equation}\label{equation:SF}
\begin{aligned}
&|u^T_\Gamma\widetilde{S}_{\Gamma} u_\Gamma-u^T_\Gamma\widetilde{S}_{\Gamma,F} u_\Gamma |\\
&=|u^T(\Atilde_{\Gamma\Gamma}-\Atilde_{\Gamma\Gamma,F})u_\Gamma
-u^T(\Atilde_{\Gamma I}A^{-1}_{II}{\Atilde^T_{\Gamma I}})u_\Gamma+u^T(\Atilde_{\Gamma I,F}A^{-1}_{II,F}{\Atilde^T_{\Gamma I,F}})u_\Gamma|\\
&\leq \left(\|\Atilde_{\Gamma\Gamma}-\Atilde_{\Gamma\Gamma,F}\|_2
+\|\Atilde_{\Gamma I}-\Atilde_{\Gamma I,F}\|_2\|A^{-1}_{II}{\Atilde^T_{\Gamma I}}\|_2\right.\\
&\qquad+\left.\|\Atilde_{\Gamma
    I,F}\|_2\|A^{-1}_{II}{\Atilde^T_{\Gamma
      I}}-A^{-1}_{II,F}{\Atilde^T_{\Gamma
      I,F}}\|_2\right)u^T_{\Gamma}u_{\Gamma}.
\end{aligned}
\end{equation}
Since
\begin{align*}
&\|A^{-1}_{II}{\Atilde^T_{\Gamma I}}-A^{-1}_{II,F}{\Atilde^T_{\Gamma I,F}}\|_2
\leq \|A^{-1}_{II}-A^{-1}_{II,F}\|_2\|{\Atilde^T_{\Gamma I}}\|_2+\|A^{-1}_{II,F}\|_2\|{\Atilde^T_{\Gamma I}}-{\Atilde^T_{\Gamma I,F}}\|_2
\end{align*}
and
\begin{equation}\label{equation:uH1}
u^T_\Gamma u_\Gamma=u^T_{\Gamma}
\widetilde{S}^{\frac12}_\Gamma
\widetilde{S}^{-1}_\Gamma
\widetilde{S}^{\frac12}_\Gamma
u_{\Gamma}
\leq \|\widetilde{S}^{-1}_\Gamma\|_2\|u_{\Gamma}\|_{\widetilde{S}_\Gamma},
\end{equation}
plugging in the \EQ{SF}, we have
\begin{align*}
&|u^T_\Gamma\widetilde{S}_{\Gamma} u_\Gamma-u^T_\Gamma\widetilde{S}_{\Gamma,F} u_\Gamma |
\leq C\gamma_{FS}\|u_\Gamma\|^2_{\widetilde{S}_{\Gamma}}.
\end{align*}
\end{proof}

\begin{lemma}\label{gammaFSinequality}
{We have the following estimates
\begin{equation}\label{equation:GammaFS}
\lim\limits_{M_{KL},M^{(i)}_{KL}\rightarrow\infty}\lim\limits_{d\rightarrow\infty} P(\{\theta:\gamma_{FS}< 1\})=1,
\end{equation}
and
\begin{equation}\label{equation:SgammaF}
\lim\limits_{M_{KL},M^{(i)}_{KL}\rightarrow\infty}\lim\limits_{d\rightarrow\infty} P(\{\theta:{u^T_{\Gamma}\widetilde{S}_{\Gamma} u_\Gamma\leq \frac1{1-\gamma_{FS}}u^T_{\Gamma}\widetilde{S}_{\Gamma,F} u_\Gamma}\})=1.
\end{equation}}
\end{lemma}
\begin{proof}
We first prove that, for any $\epsilon>0$,  there exist $N$,
$N^{(i)}$, and $M$ such that $P(\{\theta:\gamma_{FS}\ge
1\})<C\epsilon$, if $M_{KL}>N$, $M_{KL}^{(i)}>N^{(i)}$, and $d>M$.

We take $N$,
$N^{(i)}$, and $M$ satisfy Lemmas \ref{pcestimate} and
\ref{pcinverseestimate} with the bounds $\epsilon^3$.  If $M_{KL}>N$, $M_{KL}^{(i)}>N^{(i)}$, and
$d>M$, for any subdomain local matrix $A$, $E(\|A\|^2_2)$ is bounded, by \cite{CPDE2012}, and this implies that $E(\|A\|_2)$ is bounded. Let
 $\delta=\epsilon^2$ and
we have 
\begin{equation}\label{equation:deltaA}
P(\{\theta:\|A\|_2> \frac1{\sqrt{\delta}}\})\leq \sqrt{\delta}E(\|A\|_2).
\end{equation}
\begin{equation}\label{equation:expectationgammaFS}
\begin{aligned}
&P(\{\theta:\gamma_{FS}\geq 1\})\\
&\leq P\left(\{\theta:\gamma_{FS}\geq 1\}\cap\{\theta:\|\widetilde{S}^{-1}_\Gamma\|_2>\frac1{\sqrt{\delta}}\}\right)
+P\left(\{\theta:\gamma_{FS}\geq 1\}\cap\{\theta:\|\widetilde{S}^{-1}_\Gamma\|_2\leq \frac1{\sqrt{\delta}}\}\right)\\
&\leq E\left(\|\widetilde{S}^{-1}_\Gamma\|_2\right)\sqrt{\delta}+P\left(\{\theta:\gamma_{FS}\geq 1\}\cap\{\theta:\|\widetilde{S}^{-1}_\Gamma\|_2\leq \frac1{\sqrt{\delta}}\}\right)\\
&\leq C\sqrt{\delta}+P\left(\{\theta:\gamma_{FS}\geq 1\}\cap\{\theta:\|\widetilde{S}^{-1}_\Gamma\|_2\leq \frac1{\sqrt{\delta}}\}\right),
\end{aligned}
\end{equation}
where we use $E\left(\|\widetilde{S}^{-1}_\Gamma\|_2\right)$ is
bounded in the last step. 

Let the set $S_0$ include all $\theta$ such that the factors in
$\gamma_{FS}$ are bounded by $\frac1{\sqrt{\delta}}$, namely
\begin{equation}\label{equation:boundset}
 S_0= \left\{\theta:
   \|\widetilde{S}^{-1}_\Gamma\|_2\leq\frac1{\sqrt{\delta}}, \|A^{-1}_{II}{\Atilde^T_{\Gamma
        I}}\|_2\leq\frac1{\sqrt{\delta}}, \|\Atilde_{\Gamma
      I,F}\|_2\|\Atilde^T_{\Gamma I}\|_2\leq
    \frac1{\sqrt{\delta}}, \|\Atilde_{\Gamma
      I,F}\|_2\|{A^{-1}_{II,F}}\|_2\leq \frac1{\sqrt{\delta}}\right\}.
  \end{equation}

  We repeat the procedure as in \EQ{expectationgammaFS} for other
  components in $\gamma_{FS}$ and obtain
  
\begin{equation}\label{equation:expectationgammaFSall}
\begin{aligned}
&P(\{\theta:\gamma_{FS}\geq 1\})\\
&\leq E(\|\widetilde{S}^{-1}_\Gamma\|_2)\sqrt{\delta}+ E(\|A^{-1}_{II}{\Atilde^T_{\Gamma I}}\|_2){\sqrt{\delta}}
 +E(\|\Atilde_{\Gamma I,F}\|_2\|\Atilde^T_{\Gamma I}\|_2){\sqrt{\delta}}\\
 &+E(\|{A^{-1}_{II,F}}\|_2\|\Atilde_{\Gamma
   I,F}\|_2){\sqrt{\delta}}+P\left(\{\theta:\gamma_{FS}\geq 1\}\cap S_0\right)\\
 &\le C{\sqrt{\delta}}+P\left(\{\theta:\gamma_{FS}\geq 1\}\cap S_0\right).
\end{aligned}
\end{equation}

Denoting  $I_{S_0}$ as the indicator function of $S_0$ and using
\EQ{gammaFS}, we have
\begin{equation}\label{equation:secondP}
\begin{aligned}
&P(\{\theta:\gamma_{FS}\geq 1\}\cap S_0)\\
&\leq E\left(\left(\|\Atilde_{\Gamma\Gamma}-\Atilde_{\Gamma\Gamma,F}\|_2
+\|\Atilde_{\Gamma I}-\Atilde_{\Gamma I,F}\|_2\|A^{-1}_{II}{\Atilde^T_{\Gamma I}}\|_2\right.\right.\\
&+\left.\left.\|\Atilde_{\Gamma I,F}\|_2\|A^{-1}_{II}-A^{-1}_{II,F}\|_2\|{\Atilde^T_{\Gamma I}}\|_2+\|\Atilde_{\Gamma I,F}\|_2\|A^{-1}_{II,F}\|_2\|{\Atilde^T_{\Gamma I}}-{\Atilde^T_{\Gamma I,F}}\|_2\right)
\|\widetilde{S}^{-1}_\Gamma\|_2I_{S_0}\right)\\
&\le C\left(\epsilon^3+\epsilon^3\frac1{\sqrt{\delta}}+\frac1{\sqrt{\delta}}\epsilon^3 +\frac1{\sqrt{\delta}}\epsilon^3\right)\frac1{{\sqrt{\delta}}}\\
&\leq C\epsilon^3\frac{1}{\delta}=C\epsilon,
\end{aligned}
\end{equation}
where we use Lemmas \ref{pcestimate}, 
\ref{pcinverseestimate}, and {\EQ{boundset} for the last second step, and
$\delta=\epsilon^2$ for the last step.}
 Combining \EQ{expectationgammaFSall} and \EQ{secondP}, we have 
 \begin{align*}
 P(\{\theta:\gamma_{FS}>1\})&\leq
                       C{\sqrt{\delta}}+C\epsilon\le C\epsilon.
 \end{align*}
 Moreover, by Lemma \ref{lemma:SF}, we have
 $(1-\gamma_{FS})u^T_{\Gamma}\widetilde{S}_{\Gamma} u_\Gamma\le
 u^T_{\Gamma}\widetilde{S}_{\Gamma,F} u_\Gamma$ and therefore, 
\begin{align*}
&P\{\theta: u^T_{\Gamma}\widetilde{S}_{\Gamma} u_\Gamma> \frac1{1-\gamma_{FS}}u^T_{\Gamma}\widetilde{S}_{\Gamma,F} u_\Gamma\}
= P\{\theta:\gamma_{FS}> 1\}\le C\epsilon. 
\end{align*}
\end{proof}



\begin{lemma}\label{lemma:SP}
For $u_\Gamma\in\widetilde{W}_\Gamma$ and $\theta\in\Theta$,  we have
\begin{align*}
|u^T_\Gamma\widetilde{S}^{-1}_{\Gamma} u_\Gamma-u^T_\Gamma\widetilde{S}^{-1}_{\Gamma,P} u_\Gamma|\leq \gamma_{PS}\|u_{\Gamma}\|^2_{\widetilde{S}^{-1}_{\Gamma}},
\end{align*}
where 
$\gamma_{PS}=C(\gamma_1+\|S^{-1}_{\Pi}\|_2\|\Phi_P\|_2\gamma_2+\|\Phi_P\|^2_2\|S^{-1}_{\Pi}\|_2\|S^{-1}_{\Pi,P}\|_2\gamma_3)\|\widetilde{S}_\Gamma\|_2$,\\
$\gamma_1=C
\sum\limits_{i=1}^N \|{A^{(i)}_{rr}}^{-1}-A_{rr,P}^{{(i)}^{-1}}\|_2$,
$\gamma_2=C\sum\limits_{i=1}^N\|{A^{(i)}_{rr}}^{-1}{A^{(i)}_{cr}}^{T}-A_{rr,P}^{{(i)}^{-1}}{{A_{cr,P}^{{(i)}^T}\|_2}}$,
and $\gamma_3=C\sum\limits_{i=1}^N(\|A^{(i)}_{cc}-A^{(i)}_{cc,P}\|_2+\|A^{(i)}_{cr}A_{rr}^{{(i)}^{-1}}A_{cr}^{{(i)}^{T}}-A^{(i)}_{cr,P}A_{rr,P}^{{(i)}^{-1}}{{A_{cr,P}^{(i)^T}}}\|_2)$.
\end{lemma}

\begin{proof}
\begin{align*}
&|u^T_\Gamma\widetilde{S}^{-1}_{\Gamma} u_\Gamma-u^T_\Gamma\widetilde{S}^{-1}_{\Gamma,P} u_\Gamma|\\
&=|u^T_\Gamma R_{\Gamma\Delta}^T\left(\sum_{i=1}^{N}
 \left[ \vvec{0} ~ R^{(i)^T}_{\Delta}
 \right] A_{rr}^{{(i)}^{-1}}  \left[ \begin{array}{c} \vvec{0}  \\
R^{(i)}_{\Delta} \end{array} \right] \right)R_{\Gamma \Delta}u_\Gamma + u^T_\Gamma\Phi
S_{\Pi}^{-1}  \Phi^Tu_\Gamma\\
&\quad-u^T_\Gamma R_{\Gamma\Delta}^T\left(\sum_{i=1}^{N}
 \left[ \vvec{0} ~ R^{(i)^T}_{\Delta}
 \right] A_{rr,P}^{{(i)}^{-1}}  \left[ \begin{array}{c} \vvec{0}  \\
R^{(i)}_{\Delta} \end{array} \right] \right)R_{\Gamma \Delta}u_\Gamma
-u^T_\Gamma\Phi_P
S_{\Pi,P}^{-1}  \Phi^T_Pu_\Gamma|\\
&\leq |u^T_\Gamma R_{\Gamma\Delta}^T\left(\sum_{i=1}^{N}
 \left[ \vvec{0} ~ R^{(i)^T}_{\Delta}
 \right] (A_{rr}^{{(i)}^{-1}}-A_{rr,P}^{{(i)}^{-1}})  \left[ \begin{array}{c} \vvec{0}  \\
R^{(i)}_{\Delta} \end{array} \right] \right)R_{\Gamma \Delta}u_\Gamma|\\
&\quad+|u^T_\Gamma\Phi
S_{\Pi}^{-1}  \Phi^Tu_\Gamma
-u^T_\Gamma\Phi_P
S_{\Pi,P}^{-1}  \Phi^T_Pu_\Gamma|,
\end{align*}
where
\begin{align*}
&|u^T_\Gamma R_{\Gamma\Delta}^T\left(\sum_{i=1}^{N}
 \left[ \vvec{0} ~ R^{(i)^T}_{\Delta}
 \right] (A_{rr}^{{(i)}^{-1}}-A_{rr,P}^{{(i)}^{-1}})  \left[ \begin{array}{c} \vvec{0}  \\
R^{(i)}_{\Delta} \end{array} \right] \right)R_{\Gamma \Delta}u_\Gamma|\\
\leq& C
\sum\limits_{i=1}^N \|{A^{(i)}_{rr}}^{-1}-A_{rr,P}^{{(i)}^{-1}}\|_2
u^T_\Gamma u_\Gamma=\gamma_1u^T_\Gamma u_\Gamma
\end{align*}
and
\begin{align*}
&|u^T_\Gamma\Phi
S_{\Pi}^{-1}  \Phi^Tu_\Gamma
-u^T_\Gamma\Phi_P
S_{\Pi,P}^{-1}  \Phi^T_Pu_\Gamma|\\
&\leq\|\Phi-\Phi_P\|_2\|S^{-1}_{\Pi}\|_2\|\Phi\|_2u^T_\Gamma u_\Gamma
+\|\Phi_P\|_2\|S_{\Pi}^{-1}  \Phi^T-S_{\Pi,P}^{-1}  \Phi^T_P\|_2u^T_\Gamma u_\Gamma.
\end{align*}
Therefore, we need to estimate $\|\Phi-\Phi_P\|_2$ and $\|S_{\Pi}^{-1}  \Phi^T-S_{\Pi,P}^{-1}  \Phi^T_P\|_2$. We can see that 
\begin{align*}
&\|\Phi-\Phi_P\|_2
=\|R^T_{\Gamma\Delta}\sum\limits_{i=1}^N\left[ \begin{array}{c} \vvec{0}~
R^{(i)}_{\Delta} \end{array} \right]({A^{(i)}_{rr}}^{-1}{A^{(i)}_{cr}}^{T}-A_{rr,P}^{{(i)}^{-1}}A_{cr,P}^{(i)^T})R^{(i)}_{\Pi}\|_2\\
&\leq C\sum\limits_{i=1}^N\|{A^{(i)}_{rr}}^{-1}{A^{(i)}_{cr}}^{T}-A_{rr,P}^{{(i)}^{-1}}{{A_{cr,P}^{{(i)}^T}}}\|_2=\gamma_2,
\end{align*}
and
\begin{align*}
\|S_{\Pi}-S_{\Pi,P}\|_2&=\|\sum\limits_{i=1}^N(R^{(i)}_{\Pi}\{(A^{(i)}_{cc}-A^{(i)}_{cc,P})-
(A^{(i)}_{cr}A^{(i)^{-1}}_{rr}A^{(i)^T}_{cr}-
A^{(i)}_{cr,P}A^{(i)^{-1}}_{rr,P}A^{(i)^T}_{cr,P})\}R^{(i)}_{\Pi})\|_2\\
&\leq C\sum\limits_{i=1}^N(\|A^{(i)}_{cc}-A^{(i)}_{cc,P}\|_2+\|A^{(i)}_{cr}A_{rr}^{{(i)}^{-1}}A_{cr}^{{(i)}^{T}}-A^{(i)}_{cr,P}A_{rr,P}^{{(i)}^{-1}}A_{cr,P}^{{(i)}^{T}}\|_2=\gamma_3.
\end{align*}
Therefore,
\begin{align*}
&\|S^{-1}_{\Pi}-S^{-1}_{\Pi,P}\|_2=\|S^{-1}_{\Pi}(S_{\Pi}-S_{\Pi,P})S^{-1}_{\Pi,P}\|_2\\
&\leq\|S^{-1}_{\Pi}\|_2\|S^{-1}_{\Pi,P}\|_2\|S_{\Pi}-S_{\Pi,P}\|_2\leq \|S^{-1}_{\Pi}\|_2\|S^{-1}_{\Pi,P}\|_2\gamma_3,
\end{align*}
and 
\begin{align*}
    &\|S^{-1}_{\Pi}\Phi^T-S^{-1}_{\Pi,P}\Phi^T_P\|_2\\
&\leq \|S^{-1}_{\Pi}\Phi^T-S^{-1}_{\Pi}\Phi^T_P\|_2+\|S^{-1}_{\Pi}\Phi^T_P-S^{-1}_{\Pi,P}\Phi^T_P\|_2\\
&\leq \|S^{-1}_{\Pi}\|_2\|\Phi-\Phi_P\|_2+
\|S^{-1}_{\Pi}-S^{-1}_{\Pi,P}\|_2\|\Phi_P\|_2\\
&\leq \|S^{-1}_{\Pi}\|_2\gamma_2+\|\Phi_P\|_2\|S^{-1}_{\Pi}\|_2\|S^{-1}_{\Pi,P}\|_2\gamma_3.
\end{align*}
Since
\begin{equation*}
\begin{aligned}
&|u^T_\Gamma\widetilde{S}^{-1}_{\Gamma} u_\Gamma-u^T_\Gamma\widetilde{S}^{-1}_{\Gamma,P} u_\Gamma|\\
&\leq C(\gamma_1
+\|S^{-1}_{\Pi}\|_2\|\Phi_P\|_2\gamma_2+\|\Phi_P\|^2_2\|S^{-1}_{\Pi}\|_2\|S^{-1}_{\Pi,P}\|_2\gamma_3)u^T_\Gamma u_\Gamma,
\end{aligned}
\end{equation*}
and
\begin{align*}
|u^T_\Gamma u_\Gamma|&=
|u^T_\Gamma\widetilde{S}^{-\frac12}_\Gamma \widetilde{S}_\Gamma
\widetilde{S}^{-\frac12}_\Gamma
u_\Gamma|\leq \|\widetilde{S}_\Gamma\|_2
u^T_\Gamma\widetilde{S}^{-1}_\Gamma
u_\Gamma=\|\widetilde{S}_\Gamma\|_2
\|u_\Gamma\|^2_{\widetilde{S}^{-1}_\Gamma},\end{align*}
we have 
\begin{equation}\label{equation:gammaPS}
\begin{aligned}
&|u^T_\Gamma\widetilde{S}^{-1}_{\Gamma} u_\Gamma-u^T_\Gamma\widetilde{S}^{-1}_{\Gamma,P} u_\Gamma|\\
&\leq  C(\gamma_1+\|S^{-1}_{\Pi}\|_2\|\Phi_P\|_2\gamma_2+\|\Phi_P\|^2_2\|S^{-1}_{\Pi}\|_2\|S^{-1}_{\Pi,P}\|_2\gamma_3)\|\widetilde{S}_\Gamma\|_2\|u_\Gamma\|^2_{\widetilde{S}^{-1}_\Gamma}\\
&=\gamma_{PS}\|u_\Gamma\|^2_{\widetilde{S}^{-1}_\Gamma}.
\end{aligned}
\end{equation}
\end{proof}
\begin{lemma}\label{gammaPSinequality}
We have the following estimates
\begin{equation}\label{equation:GammaPS}
\lim\limits_{M_{KL},M^{(i)}_{KL}\rightarrow\infty}\lim\limits_{d\rightarrow\infty} P(\{\theta:\gamma_{PS}< 1\})=1,
\end{equation}
and
\begin{equation}\label{equation:SgammaP}
\lim\limits_{M_{KL},M^{(i)}_{KL}\rightarrow\infty}\lim\limits_{d\rightarrow\infty} P(\{\theta:u^T_\Gamma\widetilde{S}_{\Gamma,P} u_\Gamma\leq \frac1{1-\gamma_{PS}}u^T_\Gamma\widetilde{S}_{\Gamma} u_\Gamma\})=1.
\end{equation}
\end{lemma}
\begin{proof} We can prove this lemma similarly to the proof of
Lemma \ref{gammaFSinequality} using Lemmas \ref{pcestimate} and \ref{pcinverseestimate}.
\end{proof}

Based on the detailed discussion about the primal constraints and the
scaling operator choices in  \cite{Widlund:2020:BDDC}, we make the following
assumption. {Recall $E_D$ is an average operator defined in  \EQ{ED}.}
  \begin{assumption}\label{assump:ED}
    A set of primal contraints and a scaling operator $D$, selected in the BDDC
    algorithms, ensure that  $\|E_Du_\Gamma\|^2 _{\Stilde_\Gamma}\le
    C^2_{ED}\|u_\Gamma\|_{\Stilde_\Gamma}^2$, for any $u_\Gamma\in
    \Wtilde_\Gamma$. {Here $C_{ED}$ is a constant which
      is independent of $\theta$ but 
      might depend on $H$
    and $h$. {For example, in  two dimensions,
  $C^2_{ED}=C\left(1+\log\frac{H}{h}\right)^2$ when the primal
  constraints 
include the vertices of each subdomain and the cofficient does not
change too much inside each subdomain}.  }
    
  \end{assumption}

\begin{lemma}\label{lemma:PF} Under Assumption \ref{assump:PC}  and
  the conditions in Lemma \ref{pcinverseestimate}, for $u_\Gamma\in\widetilde{W}_\Gamma$, 
define $C_{PF}=\frac{1}{(1-\gamma_{PS})(1-\gamma_{FS})}$,
then we have
\begin{equation}\label{equation:CPFinequality}
\lim\limits_{M_{KL},M^{(i)}_{KL}\rightarrow\infty}\lim\limits_{d\rightarrow\infty} P(\{\theta:u^T_\Gamma\widetilde{S}_{\Gamma,P} u_\Gamma\leq  C_{PF}u^T_\Gamma\widetilde{S}_{\Gamma,F}u_\Gamma\})=1,
\end{equation}
equivalently,
\begin{equation*}
\lim\limits_{M_{KL},M^{(i)}_{KL}\rightarrow\infty}\lim\limits_{d\rightarrow\infty}
P(\{\theta:u^T_\Gamma\widetilde{S}^{-1}_{\Gamma,F} u_\Gamma\leq  C_{PF}u^T_\Gamma\widetilde{S}^{-1}_{\Gamma,P}u_\Gamma\})=1.
\end{equation*}
\end{lemma}
\begin{proof}
 Define  $A=\{\theta:u^T_\Gamma\widetilde{S}_{\Gamma} u_\Gamma\leq
 \frac1{1-\gamma_{FS}}u^T_\Gamma\widetilde{S}_{\Gamma,F} u_\Gamma\}$,
 ${B=\{\theta:u^T_\Gamma\widetilde{S}_{\Gamma,P} u_\Gamma\leq
 \frac1{1-\gamma_{PS}}u^T_\Gamma\widetilde{S}_{\Gamma}u_\Gamma\}}$, and 
 $C=\{\theta:u^T_\Gamma\widetilde{S}_{\Gamma,P} u_\Gamma\leq C_{PF}u^T_\Gamma\widetilde{S}_{\Gamma,F}\}$.
 Since $A\bigcap B\subset C$, we have $C^c\subset A^c\bigcup B^c$ and 
  \begin{align*}
 &P(\{\theta:u^T_\Gamma\widetilde{S}_{\Gamma,P} u_\Gamma> C_{PF}u^T_\Gamma\widetilde{S}_{\Gamma,F}\})\\
 &\leq P(\{\theta:u^T_\Gamma\widetilde{S}_{\Gamma} u_\Gamma>\frac1{1-\gamma_{FS}}u^T_\Gamma\widetilde{S}_{\Gamma,F} u_\Gamma\}) +P(\{\theta:u^T_\Gamma\widetilde{S}_{\Gamma,P} u_\Gamma> \frac1{1-\gamma_{PS}}u^T_\Gamma\widetilde{S}_{\Gamma} u_\Gamma\}),
 \end{align*}
which approaches 0 as $M_{KL},M^{(i)}_{KL},d$ go to infinity by \EQ{SgammaF} in Lemma \ref{gammaFSinequality} and 
\EQ{SgammaP} in Lemma \ref{gammaPSinequality}. This implies \EQ{CPFinequality}.  
  \end{proof}

\begin{lemma}\label{lemma:Ed} Under Assumption \ref{assump:ED},
for ${u_\Gamma}\in\widetilde{W}_\Gamma$, we have 
\begin{align*}
\lim\limits_{M_{KL},M^{(i)}_{KL}\rightarrow\infty}\lim\limits_{d\rightarrow\infty} P(\{\theta:\|E_D u_\Gamma\|^2_{\widetilde{S}_{\Gamma,F}}\leq C^2_{ED,F}\|u_{\Gamma}\|^2_{\widetilde{S}_{\Gamma,F}}\})=1,
\end{align*}
where $C^2_{ED,F}=\frac{1+\gamma_{FS}}{1-\gamma_{FS}} C^2_{ED}$.
\end{lemma}
\begin{proof}
{Taking $u_\Gamma$ in Lemma \LA{SF} to be $E_Du_\Gamma $ and using Assumption \ref{assump:ED},} we have
\begin{align*}
\|E_Du_\Gamma\|^2_{\widetilde{S}_{\Gamma,F}}
\leq (1+\gamma_{FS})\|E_Du_\Gamma\|^2_{\widetilde{S}_{\Gamma}}
\leq(1+\gamma_{FS}) C^2_{ED}
\|u_\Gamma\|^2_{\widetilde{S}_\Gamma},
\end{align*}
then
\begin{align*}
&P(\{\theta:\|E_D u_\Gamma\|^2_{\widetilde{S}_{\Gamma,F}}> C^2_{ED,F}\|u_{\Gamma}\|^2_{\widetilde{S}_{\Gamma,F}}\})\\
&\leq P(\{\theta:(1+\gamma_{FS})C^2_{ED}\|u_\Gamma\|^2_{\widetilde{S}_\Gamma}> C^2_{ED,F}
\|u_\Gamma\|^2_{\widetilde{S}_{\Gamma,F}}\})\\
&=P(\{\theta:\|u_\Gamma\|^2_{\widetilde{S}_\Gamma}>\frac1{1-\gamma_{FS}}
\|u_\Gamma\|^2_{\widetilde{S}_{\Gamma,F}}\}),
\end{align*}
which goes to 0 as $M_{KL},M^{(i)}_{KL},d$ approach infinity by Lemma \ref{gammaFSinequality}. This implies the result.

\end{proof}

\begin{theorem}\label{theorem:cond}
For $u_\Gamma\in\widehat W_\Gamma,$
the condition number of the preconditioned operator $M^{-1}F$ follows 
$$\lim\limits_{M_{KL},M^{(i)}_{KL}\rightarrow\infty}\lim\limits_{d\rightarrow\infty} P(\{\theta:\frac{1}{C_{PF}}u^T_\Gamma Mu_\Gamma\leq
u_\Gamma Fu_\Gamma\leq CC_{ED,F}^2C_{PF}
u^T_\Gamma Mu_\Gamma\})=1,$$
where the preconditioner 
$M^{-1}=\Rtilde^T_{D,\Gamma}\Stilde^{-1}_{\Gamma,P}\Rtilde_{D,\Gamma}$
and the operator
$F=S_{\Gamma,F}=\Rtilde_\Gamma^T\Stilde_{\Gamma,F}\Rtilde_\Gamma$,
the constants $C_{ED,F}$ and $C_{PF}$ are defined in Lemmas
\LA{Ed}, and \LA{PF}, respectively.
\end{theorem}
\begin{proof}

{\bf Lower bound:}
Let 
\begin{equation}\label{equation:wgamma}
w_{\Gamma}=Mu_\Gamma=\left(\Rtilde_{D,\Gamma}^T\Stilde_{\Gamma,P}^{-1}\Rtilde_{D,\Gamma}\right)^{-1}u_\Gamma
\in  \What_{\Gamma}.
\end{equation}
Using the properties $\Rtilde^T_\Gamma \Rtilde_{D,\Gamma}=\Rtilde^T_{D,\Gamma}\Rtilde_\Gamma=I$ and \EQ{wgamma}, 
we have,
\begin{eqnarray}
&&u_\Gamma^TMu_\Gamma
=
u_\Gamma^T \left(\Rtilde_{D,\Gamma}^T\Stilde_{\Gamma,P}^{-1}\Rtilde_{D,\Gamma}\right)^{-1}u_\Gamma=u_\Gamma^Tw_\Gamma\nonumber\\
&=&u_\Gamma^T\Rtilde_\Gamma^T \Stilde_{\Gamma,P}\Stilde_{\Gamma,P}^{-1}\Rtilde_{D,\Gamma} w_\Gamma
=\left<\Rtilde_\Gamma u_\Gamma,\Stilde_{\Gamma,P}^{-1}\Rtilde_{D,\Gamma} w_\Gamma\right>_{\Stilde_{\Gamma,P}}\nonumber\\
&\le&\left<\Rtilde_\Gamma u_\Gamma,\Rtilde_\Gamma u_\Gamma\right>^{1/2}_{\Stilde_{\Gamma,P}}\left<\Stilde_{\Gamma,P}^{-1}\Rtilde_{D,\Gamma} w_\Gamma,\Stilde_{\Gamma,P}^{-1}\Rtilde_{D,\Gamma} w_\Gamma\right>^{1/2}_{\Stilde_{\Gamma,P}}\nonumber\\
&=&\left(u_\Gamma^T\Rtilde_\Gamma^T\Stilde_{\Gamma,P} \Rtilde_\Gamma u_\Gamma\right)^{1/2}\left(w_\Gamma^T\Rtilde_{D,\Gamma}^T\Stilde_{\Gamma,P}^{-1}\Stilde_{\Gamma,P}\Stilde_{\Gamma,P}^{-1}\Rtilde_{D,\Gamma}w_\Gamma\right)^{1/2}\nonumber\\
&=&\left(u_\Gamma^T\Rtilde_\Gamma^T\Stilde_{\Gamma,P} \Rtilde_\Gamma u_\Gamma\right)^{1/2}\left(u_\Gamma^TMu_\Gamma\right)^{1/2}.
\end{eqnarray}
{Squaring both sides and cancelling the common factor, we have
$$u_\Gamma^TMu_\Gamma\le u_\Gamma^T\Rtilde_\Gamma^T\Stilde_{\Gamma,P} \Rtilde_\Gamma u_\Gamma.$$}
Therefore,
{
\begin{align*}
&P(\{\theta:u^T_\Gamma Mu_\Gamma>C_{PF}u^T_\Gamma Fu_\Gamma\})\\
&\leq P(\{\theta:u_\Gamma^T\Rtilde_\Gamma^T\Stilde_{\Gamma,P} \Rtilde_\Gamma u_\Gamma> C_{PF}u^T_\Gamma Fu_\Gamma\})\\
&=P(\{\theta:u_\Gamma^T\Rtilde_\Gamma^T\Stilde_{\Gamma,P} \Rtilde_\Gamma u_\Gamma> C_{PF}u_\Gamma^T\Rtilde_\Gamma^T\Stilde_{\Gamma,F} \Rtilde_\Gamma u_\Gamma\}),
\end{align*}
which approaches 0, as $M_{KL},M^{(i)}_{KL},d$ go to infinity, by Lemma \ref{lemma:PF}.
Then we have 
\begin{align*}
\lim\limits_{M_{KL},M^{(i)}_{KL}\rightarrow\infty}\lim\limits_{d\rightarrow\infty} P(\{\theta:\frac 1{C_{PF}}u^T_\Gamma Mu_\Gamma\leq u^T_\Gamma Fu_\Gamma\})=1.
\end{align*}
}

{\bf Upper bound:}
Using the definition of $w_\Gamma$ in \EQ{wgamma}, the Cauchy-Schwarz inequality, and Lemma \LA{Ed}, we obtain the upper 
bound:
\begin{eqnarray}
&&u_\Gamma^TF u_\Gamma
=u_\Gamma^T\Rtilde_\Gamma^T\Stilde_{\Gamma,F}\Rtilde_\Gamma\Rtilde_{D,\Gamma}^T\Stilde_{\Gamma,F}^{-1}\Rtilde_{D,\Gamma}w_\Gamma\nonumber\\
&=&\left<\Rtilde_{\Gamma}u_\Gamma,E_{D}\Stilde_{\Gamma,F}^{-1}\Rtilde_{D,\Gamma}w_\Gamma\right>_{\Stilde_{\Gamma,F}}\nonumber\\
&\le&\left<\Rtilde_{\Gamma}u_\Gamma,\Rtilde_{\Gamma}u_\Gamma\right>^{1/2}_{\Stilde_{\Gamma,F}}\left<E_{D}\Stilde_{\Gamma,F}^{-1}\Rtilde_{D,\Gamma}w_\Gamma,E_{D}\Stilde_{\Gamma,F}^{-1}\Rtilde_{D,\Gamma}w_\Gamma\right>^{1/2}_{\Stilde_{\Gamma,F}}\nonumber\\
&=&(u_\Gamma^TF u_\Gamma)^{\frac12}\left<E_{D}\Stilde_{\Gamma,F}^{-1}\Rtilde_{D,\Gamma}w_\Gamma,E_{D}\Stilde_{\Gamma,F}^{-1}\Rtilde_{D,\Gamma}w_\Gamma\right>^{1/2}_{\Stilde_{\Gamma,F}}.
\end{eqnarray}
{Squaring both sides and cancelling the common factor, we have
  \begin{equation}\label{equation:upper}
      u_\Gamma^TF u_\Gamma\le \|E_D\Stilde_{\Gamma,F}^{-1}\Rtilde_{D,\Gamma}w_\Gamma\|^2_{\widetilde{S}_{\Gamma,F}}.\end{equation}}
Therefore, similar as the proof in Lemma \ref{lemma:PF}, we define
\begin{align*}
  A&=\{\theta:\|E_D\Stilde_{\Gamma,F}^{-1}\Rtilde_{D,\Gamma}w_\Gamma\|^2_{\widetilde{S}_{\Gamma,F}}\leq C^2_{ED,F}{\|\Stilde_{\Gamma,F}^{-1}\Rtilde_{D,\Gamma}w_\Gamma\|}^2_{\Stilde_{\Gamma,F}}\},\\
  B&=\{\theta:{\|\Stilde_{\Gamma,F}^{-1}\Rtilde_{D,\Gamma}w_\Gamma\|}^2_{\Stilde_{\Gamma,F}}\leq C_{PF}u_{\Gamma}^TMu_{\Gamma}\},\\
 C&=\{\theta:\|E_D\Stilde_{\Gamma,F}^{-1}\Rtilde_{D,\Gamma}w_\Gamma\|^2_{\widetilde{S}_{\Gamma,F}}\leq
    C^2_{ED,F}C_{PF}u_{\Gamma}^TMu_{\Gamma}\}.
\end{align*}
We have $A\bigcap B\subset C$ and  then $C^c\subset A^c\bigcup B^c$. 
Using \EQ{upper}, we have 
\begin{align*}
&P(\{\theta:u^T_\Gamma Fu_\Gamma>C^2_{ED,F}C_{PF}u_{\Gamma}^TMu_{\Gamma}\})\\
&\leq P(\{\theta:\|E_D\Stilde_{\Gamma,F}^{-1}\Rtilde_{D,\Gamma}w_\Gamma\|^2_{\widetilde{S}_{\Gamma,F}}> C^2_{ED,F}C_{PF}u_{\Gamma}^TMu_{\Gamma}\})\\
&\leq P(\{\theta:{\|E_D\Stilde_{\Gamma,F}^{-1}\Rtilde_{D,\Gamma}w_\Gamma\|^2_{\widetilde{S}_{\Gamma,F}}> C^2_{ED,F}{\|\Stilde_{\Gamma,F}^{-1}\Rtilde_{D,\Gamma}w_\Gamma\|}^2_{\Stilde_{\Gamma,F}}}\})\\
&\quad+P(\{\theta:{\|\Stilde_{\Gamma,F}^{-1}\Rtilde_{D,\Gamma}w_\Gamma\|}^2_{\Stilde_{\Gamma,F}}> C_{PF}u_{\Gamma}^TMu_{\Gamma}\}).
\end{align*}
The first term goes to 0, as $M_{KL},M^{(i)}_{KL},d$ approach infinity by Lemma \ref{lemma:Ed}. For the second term, since 
\begin{align*}
  &{\|\Stilde_{\Gamma,F}^{-1}\Rtilde_{D,\Gamma}w_\Gamma\|}^2_{\Stilde_{\Gamma,F}}=w_\Gamma^T\Rtilde_{D,\Gamma}^T\Stilde_{\Gamma,F}^{-1}\Stilde_{\Gamma,F}\Stilde_{\Gamma,F}^{-1}\Rtilde_{D,\Gamma}w_\Gamma=u_\Gamma^TM\left(\Rtilde_{D,\Gamma}^T\Stilde_{\Gamma,F}^{-1}\Rtilde_{D,\Gamma}\right)Mu_\Gamma\\
 & \mbox{and}\\
&u_\Gamma^TM\left(\Rtilde_{D,\Gamma}^T\Stilde_{\Gamma,P}^{-1}\Rtilde_{D,\Gamma}\right)Mu_\Gamma=u^T_\Gamma M
u_\Gamma,
\end{align*}
we have
\begin{align*}
&P(\{\theta:{\|\Stilde_{\Gamma,F}^{-1}\Rtilde_{D,\Gamma}w_\Gamma\|}^2_{\Stilde_{\Gamma,F}}> C_{PF}u_{\Gamma}^TMu_{\Gamma}\})\\
&=P(\{\theta:u_\Gamma^TM\left(\Rtilde_{D,\Gamma}^T\Stilde_{\Gamma,F}^{-1}\Rtilde_{D,\Gamma}\right)Mu_\Gamma>
u_\Gamma^TM\left(\Rtilde_{D,\Gamma}^T\Stilde_{\Gamma,P}^{-1}\Rtilde_{D,\Gamma}\right)Mu_\Gamma\})\\
&=P\left(\left\{\theta:\left(\Rtilde_{D,\Gamma} Mu_\Gamma\right)^T\Stilde_{\Gamma,F}^{-1}\left(\Rtilde_{D,\Gamma}Mu_\Gamma\right)>
\left(\Rtilde_{D,\Gamma} Mu_\Gamma\right)^T\Stilde_{\Gamma,P}^{-1}\left(\Rtilde_{D,\Gamma}Mu_\Gamma\right)\right\}\right),
\end{align*}
which goes to 0, as $M_{KL},M^{(i)}_{KL},d$ go to infinity by  Lemma \ref{lemma:PF}.
This implies that 
\begin{align*}
\lim\limits_{M_{KL},M^{(i)}_{KL}}\lim\limits_{d\rightarrow\infty}P(\{\theta:u^T_\Gamma Fu_\Gamma\leq C^2_{ED,F}C_{PF}u_{\Gamma}^TMu_{\Gamma}\})=1.
\end{align*}
\end{proof}
\section{Numerical Experiments}

In this section, we study the performance of our stochastic BDDC
preconditioners and {{compare it}} with the deterministic  exact BDDC
preconditioners and the mean-based BDDC preconditioners.  
Our computational domain
{$\Omega=[0,1]\times [0,1]$} is first decomposed to $N=N_s\times N_s$
subdomains. Each subdomain is further splited to $2n\times n$ uniform
triangles.  We denote $H=\frac{1}{N_s}$ and $h=\frac{H}{n}$. In our
experiments, 
$f={2\pi^2\sin(\pi x)\sin(\pi y)}$ and \EQ{pde1} is discretized by  piecewise
linear finite elements in space and 
the discretization of the coefficient $\kappa$ is a piecewise 
constant approximation.  

{ In our numerical experiments, the coefficient of
  \EQ{pde1} $\kappa (\x,
  \theta)=\exp(a(\x,\theta))$, where $a(\x,\theta)\sim N(0,C)$ and $C$
  is defined in \EQ{cov}. We first do a global KL
  truncation to obtain the global approximation $a_{M_{KL}}$ defined in
  \EQ{globalKL}.  A single-level Monte Carlo sampling method is used.
  We generate  samples of $\xi_1, \cdots, \xi_{M_{KL}}$ in
  \EQ{globalKL} and obtain the samples of $a_{M_{KL}}$.  The  Schur complement \EQ{globalschur} resulting from
\EQ{pde1} with $\kappa(\x,\theta)\approx a_{M_{KL}}$ is solved by preconditioned conjugate gradient methods (PCG)
with different BDDC preconditioners.  The solutions provide samples of
the solutions of \EQ{pde1}.   The CG is stopped when the $l_2$
norm of the relative residual has been reduced by a factor of $10^{-8}$ {{or when it 
reaches the maximum iteration number $100$}}.}

{In those
  experiments where our stochastic BDDC components are  only used as a preconditioner, the
  Schur complements are obtained from \EQ{pde1} using $a_{M_{KL}}$,  and
  therefore the solutions are good approximation of those of
  \EQ{pde1} with {{a large}} enough $M_{KL}$.   We also apply
our stochastic BDDC preconditioner  {to the approximate}  Schur
complement $S_{\Gamma,F}$, which {{was}} studied  in \cite{Contreras2018B} without
any preconditioners.  The average relative $L_2$ errors, compared with the solution
of \EQ{pde1}, are provided in Tables \ref{Tab:inexactFcN} to 
\ref{Tab:inexactFcPL}. }

{We note that, for the stochastic BDDC algorithms, the
  expensive constructions are performed in the offline stage, {{most of which}} are local to each subdomain and can be computed in parallel. During
  the online stage, the constructions  of the local components of the
  algorithms are local to each subdomain and can be computed in
  parallel. The coarse problem, the only global component in the
  algorithms, is constructed similarly to those  in the
deterministic BDDC algorithms as long as the subdomain local
components are available, which are mainly constructed in the {{offline}}
stage for the stochastic BDDC algorithms.  Therefore, the only possible difference {{in}}
the weak scalability between the stochastic and deterministic BDDC
algorithms is the  CG iterations, since the stochastic BDDC
preconditioner is  an approximation.  The weak scalability of the
deterministic BDDC algorithms {{has}} been tested in many different
applications, such as \cite{SSM2013,Zampini2016,ZT:2016:Darcy,HSB2020,VFEBDDC2022}. To
check the weak scalability of our stochastic BDDC algorithms, we 
present the condition numbers and/or CG iterations in our numerical
experiments.}

{In all experiments,  the subdomain
vertex {{constraints}} are chosen, {{and}} a simple  $\rho$-scaling is used.
For  the problems with high contrast coefficients  inside the subdomains, we will
need more advanced  adaptive primal constraints, which are constructed
by solving local eigenvalue problems,  and the deluxe scaling \cite{ZT:2016:Darcy,KKR2020,ZTBrinkman2022,STX2021}. 
We will work on constructing these additional primal constraints using
our stochastic approximations of the local matrices in our
future work. 
}
{{We have four sets of {{ numerical experiments}}. In each, we draw $100$ samples
and report the average number of CG iterations.}}  We also use the same number
of subdomain local KL terms in each subdomain. We will denote
``Exact'' {{as}} the exact BDDC preconditioner, ``MPC'' {{as}} the
mean-valued BDDC preconditioner, and ``Stochastic'' {{as}} our proposed
stochastic BDDC preconditioner.  In the first three sets, we show the
results with both SG and SC constructions. These two constructions
provide similar results.
In our fourth set for the inexact
Schur complement, we only show the results for the SG {{construction}}
since the SC results are very similar. 

{\bf Set I:} In our first set  of the experiments, we take $l=1$ and $\sigma^2=0.5$
in \EQ{cov} and the coefficient $\kappa$ is smooth. {We keep the first
$4$ 
global KL terms and 
$M_{KL}=4$. The percentage of the first four eigenvalues over all
eigenvalues is $98.2\%$. } We report both the
condition numbers of the preconditioned systems and the number of the PCG iterations.
As {{with}} standard BDDC algorithms, we test the
performance of different BDDC algorithms {{by}} increasing 
{the number of subdomains} and the subdomain local problem sizes.  We
first fix $\frac{H}{h}=8$ and the local {{PC}} degree
$d=4$.  We change the number of subdomains $N_s$ and the subdomain local KL
term $N_{KL}$.   The results are reported in Table
\ref{Tab:ex1-1}.  For the ``Exact'' preconditioners, the condition
numbers and the number of PCG iterations are independent of $N_s$, as
we expect from the BDDC theory \cite{Widlund:2020:BDDC}. For the
``MPC'' preconditioners,  the condition numbers and the number of the
CG iterations {{both increases as $N_s$ increases}}.  For
our stochastic BDDC preconditioners, the condition numbers and the CG
iterations {{increase slightly as $N_s$ increases}} for
all choices of $N_{KL}$.  However, {{these numbers are much closer to those obtained with the
exact BDDC preconditioners than to those for ``MPC''. }} {{We note
that as $N_s$ increases, the size of the subdomain {{becomes}}
smaller, and the percentage of  the largest eigenvalue $\lambda_1^{(i)}$ in
\EQ{LML} over all eigenvalues rises from $98.0\%$ to $99.9\%$. Consequently,
{{a}} small number of the KL terms can provide {{an}} accurate approximation.}}
As a result, for
$N_s=16\times 16$, the condition numbers and the CG iterations are
almost identical for $N_{KL}=1,2,3$.  We found that, for large $N_s$,
the SG construction provides better results than the SC construction. For deterministic BDDC
algorithms, the scalability (the CG iterations are independent of
$N_s$) is important {{for parallel computation}}.
Since one
subdomain is {usually} assigned to one processor, 
the scalability means that the CG iterations are independent of the
number of the processors.  When we run larger problems, we can use
more processors with {a} similar number of the CG {{iterations.}}
 {{For the stochastic BDDC algorithms, a good approximation requires a larger number of subdomains, a smaller subdomain problem size, and fewer subdomain local KL terms.}} However, with {{larger}} $N_s$, the size of the Schur
complement {becomes} larger, {{and}} the cost of the preconditioners in \cite{stochasticpRMCM2021}
increases considerably. In our stochastic BDDC algorithms, we {{do not}}
form the global  approximation of the Schur complements. {{Only the size of the coarse problem, which is much smaller than the global Schur complement, increases with increasing Ns.}}
When the
number of processors/subdomains is very large,
the coarse problem might become a bottle-neck and three or multilevel
BDDC algorithms can be used to remove the bottle-neck
\cite{Tu:2004:TLB,Tu:2005:TLB,Tu:2008:TBS,ZT:2016:Darcy}.
In Figure
1, we plot the number of the CG iterations for  ``MPC'',
stochastic, and exact BDDC preconditioners for $N_s=256$ with 100 samples.

\begin{figure}[h]
\begin{center}
\caption{\label{figure:samples} The CG iterations for $100$ samples
  with  $N_s=256$, $\frac{H}{h}=8$, $N_{KL}=1$, $d=4$, $l=1$,
$\sigma^2=0.5$..}
\includegraphics[scale=0.3]{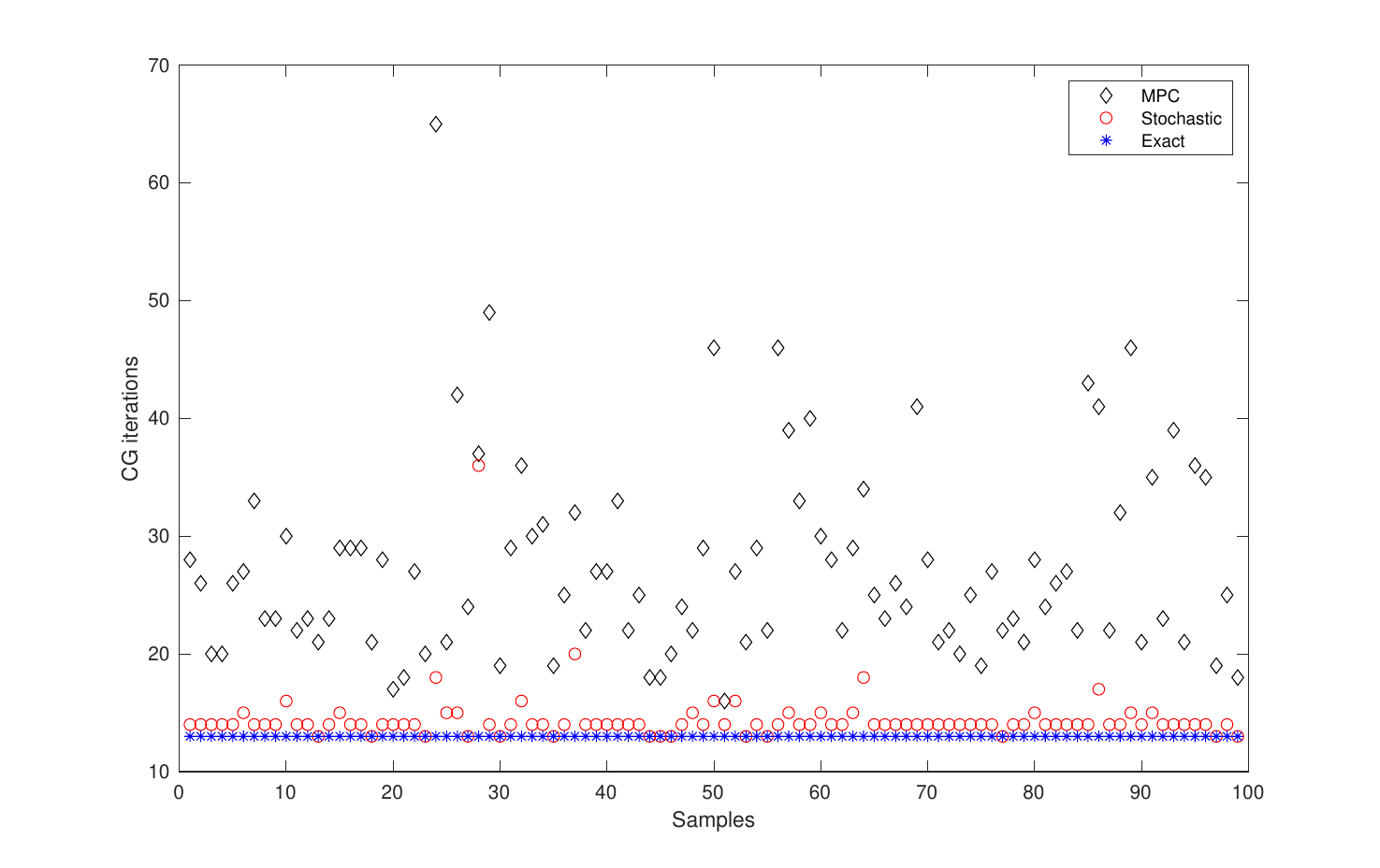}
\end{center}
\end{figure}

To further improve the performance of our stochastic BDDC algorithms, we keep $N_{KL}=1$ and increase the PC polynomial
degree $d$ from $4$ to $6$, the results are reported in Table
\ref{Tab:exp-1}. {{As $d$ increases}},  the condition numbers
and the CG iterations of our stochastic BDDC {algorithms} further decrease to {{be}}
almost identical to those with the exact BDDC preconditioners.  The
stochastic BDDC algorithms are scalable for $d\ge 5$.

In Tables \ref{Tab:ex1-2} and \ref{Tab:exp-2}, we study the performance
of different BDDC preconditioners with  a change of the subdomain
problem size $\frac{H}{h}$. The number of the subdomains is fixed {{at}}
$64$ and the percentage of  the largest eigenvalue $\lambda_1^{(i)}$ in
\EQ{LML} over all eigenvalues is $99.5\%$. One KL term provides a good
approximation of the local coefficient.  By the BDDC theory
\cite{Widlund:2020:BDDC}, the {{number}} of CG iterations increase slowly
{{as $\frac{H}{h}$ increases}}.  {{In contrast to Table
\ref{Tab:ex1-1}}}, we can observe the decreases of both condition number
and the CG iterations {{as $N_{KL}$} increases} in Table
\ref{Tab:ex1-2}.   Similar to Table \ref{Tab:exp-1}, the performance
of the stochastic BDDC algorithms {{is also improved by increasing d.}} 

All these results are consistent with our theorem.

{\bf Set II:} In our second set of {{numerical}} experiments, we change $\sigma^2$ in
\EQ{cov}. The increase of $\sigma^2$ does not affect the global and local KL
{{expansions}}, but {{increases}} the variability of $\kappa$. We note that
for ``MPC'', the CG iterations might not converge in $100$ iterations for
some samples and
the condition number estimates are not meaningful anymore. Therefore, we only
report the number of CG iterations from now on. We report the results
in Table \ref{Tab:sig} for $64$ subdomains and $\frac{H}{h}=8$ with
$N_{KL}=3$ and $d=6$.  The performance of the Stochastic BDDC
algorithms {{is}} still close to the exact BDDC algorithms. But the MPC
{{deteriorates}} (recall that we set the maximum numbers of the iterations
to be $100$).

{\bf Set III:} In our third set of the numerical experiments, we further increase the
variability of the samples by decreasing $l$ in \EQ{cov}.  In this
set, $l=0.1$. {We keep the first
$15$ 
global KL terms and 
$M_{KL}=15$. The percentage of the first four eigenvalues over all
eigenvalues is $95.8\%$.} We repeat the numerical experiments in our first
set. {{
For small $l$ and 16 subdomains, the first two largest eigenvalues ($N_{KL}=2$) from \EQ{LML}  account for only $90.7\%$ of the total eigenvalues. This percentage increases to $99.3\%$ when $N_{KL}=4$.}} Therefore, in Table
\ref{Tab:exL1-1}, {{we can see the
decrease of the number of the CG iterations.}} For $N_s\ge
8$, $N_{KL}=3$ gives $99.9\%$, we cannot see the decease {{of the number of the
CG iterations}} from $N_{KL}=3$ to $N_{KL}=4$ anymore. 
{In Table
\ref{Tab:exLp-3}, we increase the polynomial degrees $d$
with $N_{KL}=3$.}  The performance of our
stochastic BDDC are very close to those of the exact BDDC.  For
$N_s=16$ and $d=5$, we found {{that}} the resulting $S_\Pi$ is not positive
definite for one sample. We use the rest sample results to calculate
the average iterations. The result
is marked with $*$, see Remark \ref{remark:Cspd} for more details.

{
We also plot the average Frobenius norm of the difference between the
exact $S_{\Pi}$ and our stochastic approximation $S_{\Pi,PC_d}$ in
Figure \ref{fig:ScErr}.  These errors are comparable with those
obtained in \cite{Contreras2018B} for the Schur complement error.  The top figure in Figure \ref{fig:ScErr}
shows that the error decreases with the increasing of the number of the local KL
terms $N_{KL}$. {{The
change of the error becomes smaller when $N_{KL}$ increases from $3$ to $4$ compared to the change}}
from $2$ to $3$. This is consistent with the iteration changes in the
last row of Table \ref{Tab:exL1-1}. The bottom  figure in Figure \ref{fig:ScErr}
shows that the error decreases {{as the PC
polynomial degree $d$ increases. The
change of the error becomes smaller when $d$ increases from $5$ to $6$ compared to the change
from $4$ to $5$.}} This is also consistent with the iteration changes in the
last row of Table \ref{Tab:exLp-3}.  Also when $d$ changes from $5$ to
$6$, the approximate coarse problem of that ``bad'' sample  changes from
indefinite to positive
definite. }

\begin{figure}[h]
\begin{center}
\caption{\label{fig:ScErr} The $\|S_{\Pi}-S_{\Pi,PC_d}\|_F$ for
  the mean of $100$ samples
  with  $N_s=256$, $\frac{H}{h}=8$, $l=0.1$,
$\sigma^2=0.5$. The top figure is for $N_{KL}=2,3,4$ and $d=4$
(corresponding to 
the last row in  Table \ref{Tab:exL1-1}); the bottom
figure is for $d=3,4,5$ and $N_{KL}=3$ (corresponding to 
the last row in  Table \ref{Tab:exLp-3}).  }
\includegraphics[scale=0.5]{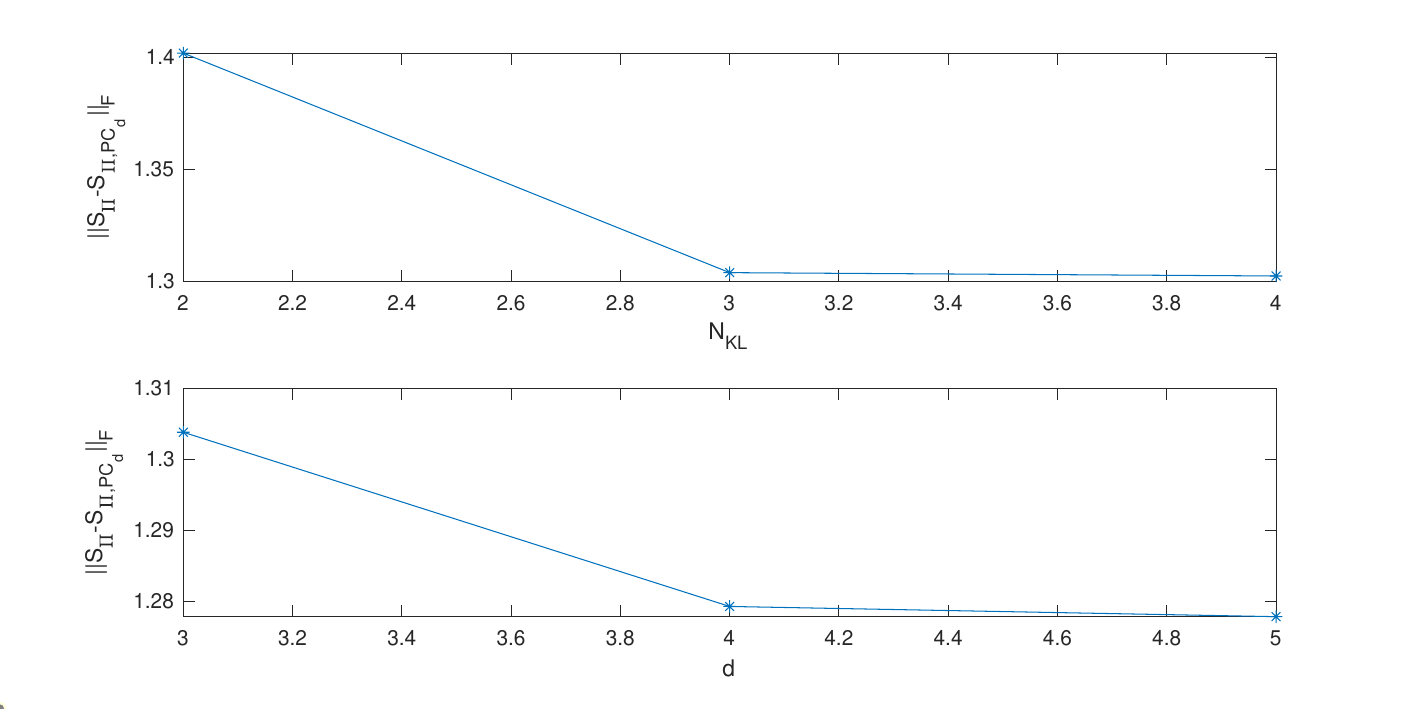}
\end{center}
\end{figure}

In Tables \ref{Tab:exL1-2} and \ref{Tab:exLp-4}, we
study the {{performance}} of those algorithms with a change of the subdomain
problem size $\frac{H}{h}$. We can observe similar performance {{in Tables \ref{Tab:exL1-1}, \ref{Tab:exLp-3}.}}

{\bf Set IV:}   In our fourth set of the numerical experiments, we
test the performance of our stochastic BDDC {algorithms} to precondition
the approximate Schur complement  $S_{\Gamma,F}$.  We fix the number of the
subdomains {{as}} $64$ and change the values of the KL terms $N_{KL}$ and the
polynomial degree $d$ for $l=1$ and $l=0.1$. The results are presented
in Tables \ref{Tab:inexactFcN} to \ref{Tab:inexactFcPL}.  {{The results
show that our stochastic BDDC algorithms work well with the
approximate Schur complement.}} Here we only present the results
with the SG construction{{,}} since the SC construction gives similar
results.  Moreover, we also provide the average $L_2$ error of the
solutions obtained using the approximated $S_{\Gamma,F}$. 

\begin{table}
\caption{Performance of Stochastic BDDC preconditioner  for a change
  of the number of the subdomains and $N_{KL}$ with
  $\frac{H}{h}=8$, $\sigma^2=0.5$, $l=1$, $d=4$. }\label{Tab:ex1-1}
\begin{center}
\begin{tabular}{|c|cc||cc|cc|cc||cc|}\hline\hline
&\multicolumn{2}{|c||}{MPC}&\multicolumn{2}{|c|}{ $N_{KL}=1$}&\multicolumn{2}{|c|}{ $N_{KL}=2$}
                                                  & \multicolumn{2}{c||}{ $N_{KL}=3$} & \multicolumn{2}{c|}{ Exact}\\ \hline
$N_s$ & Cond. & Iter. & Cond. & Iter.& Cond. & Iter. & Cond. & Iter.   & Cond. & Iter.  \\ \hline
&\multicolumn{2}{|c||}{}&\multicolumn{6}{|c||}{ with the SC construction}& \multicolumn{2}{|c|}{ }\\ \hline
4  &5.01&18.25&2.54&11.98&2.45& 11.73 &2.29&10.45  &2.22&10.51 \\
8   &7.12& 23.58&2.70&13.55&2.66& 13.43&2.62&13.20 &2.45&12.29\\
 12  & 7.97&25.99&2.89&14.10 &2.89&13.99 &2.85&13.85& 2.50&13.00 \\
16 &8.37& 27.48&3.15&14.44&3.15& 14.37&3.15&14.30&2.49 & 13.00\\\hline
&\multicolumn{2}{|c||}{}&\multicolumn{6}{|c||}{ with the SG construction}& \multicolumn{2}{|c|}{ }\\ \hline
4  &5.01&18.25&2.72&13.18&2.62& 12.51 &2.42&10.73  &2.22&10.51 \\
8  &7.12& 23.58 &2.65&12.85&2.61& 12.39&2.54&11.36 &2.45&12.29\\
12 & 7.97&25.99&2.63&12.66 &2.61&12.31 &2.58&11.48& 2.50&13.00 \\
16 &8.37& 27.48 &2.63&12.58&2.61& 12.26&2.59&11.61&2.49 & 13.00\\
  \hline
\end{tabular}
\end{center}
\end{table}
\begin{table}
\caption{Performance of Stochastic BDDC preconditioner  for  for a change
  of the number of the subdomains and $d$ with
  $\frac{H}{h}=8$, $\sigma^2=0.5$, $l=1$, $N_{KL}=1$. }\label{Tab:exp-1}
\begin{center}
\begin{tabular}{|c|cc||cc|cc|cc||cc|}\hline\hline
&\multicolumn{2}{|c||}{MPC}&\multicolumn{2}{|c|}{ $d=4$}&\multicolumn{2}{|c|}{ $d=5$}
                                                    &\multicolumn{2}{c||}{ $d=6$}& \multicolumn{2}{|c|}{ Exact}\\ \hline
$N_s$ & Cond. & Iter.& Cond. & Iter.  & Cond. & Iter. & Cond. & Iter.   & Cond. & Iter.  \\ \hline
&\multicolumn{2}{|c||}{}&\multicolumn{6}{|c||}{ with the SC construction}& \multicolumn{2}{|c|}{ }\\ \hline
 4  &5.01&18.25&2.54&11.98&2.53& 11.95 &2.53&11.95  &2.22&10.51 \\
8   &7.12& 23.58&2.70&13.55&2.59& 13.30&2.57&13.25 &2.45&12.29\\
12  & 7.97&25.99&2.89&14.10 &2.62&13.45 &2.56&13.34& 2.50&13.00 \\
16 &8.37& 27.48&3.15&14.44&2.65& 13.61&2.56&13.29&2.49 & 13.00\\
  \hline
  &\multicolumn{2}{|c||}{}&\multicolumn{6}{|c||}{ with the SG construction}& \multicolumn{2}{|c|}{ }\\ \hline
4  &5.01&18.25&2.72&13.18&2.72& 13.16 &2.72&13.17  &2.22&10.51  \\
8   &7.12& 23.58&2.65&12.85&2.63& 12.72&2.63&12.70 &2.45&12.29\\
12  & 7.97&25.99&2.63&12.66&2.59&12.35 &2.59&12.33& 2.50&13.00 \\
  16 &8.37& 27.48&2.63&12.58&2.57& 12.08&2.57&12.07&2.49 & 13.00\\\hline
\end{tabular}
\end{center}
\end{table}

\begin{table}
\caption{Performance of Stochastic BDDC preconditioner   for a change
  of the subdomain local problem size $\frac{H}{h}$ and {$N_{KL}$} with $64$
  subdomains, $\sigma^2=0.5$, $l=1$, $d=4$. }\label{Tab:ex1-2}
\begin{center}
\begin{tabular}{|c|cc||cc|cc|cc||cc|}\hline\hline
&\multicolumn{2}{|c||}{MPC}&\multicolumn{2}{|c|}{ $N_{KL}=1$}&\multicolumn{2}{|c|}{ $N_{KL}=2$}
                                                    &\multicolumn{2}{|c||}{ $N_{KL}=3$} & \multicolumn{2}{|c|}{ Exact}\\ \hline
$\frac{H}{h}$ &Cond. & Iter.& Cond. & Iter.  & Cond. & Iter. & Cond. & Iter.   & Cond. & Iter.  \\ \hline
&\multicolumn{2}{|c||}{}&\multicolumn{6}{|c||}{ with the SC construction}& \multicolumn{2}{|c|}{ }\\ \hline
4  &5.18&19.94 &1.95&10.83&1.94&10.73 &1.92&10.54 &1.78&9.93 \\
8   &7.12& 23.58&2.70&13.55&2.66& 13.43&2.62&13.20 &2.45&12.29\\
12 &8.47&25.73 &3.23&15.32&3.22&15.32 &3.16&14.90&2.92 &13.48 \\
16 &9.53&27.29&3.64&16.41&3.61& 16.13&3.56&15.70& 3.28& 14.98\\
 \hline
&\multicolumn{2}{|c||}{}&\multicolumn{6}{|c||}{ with the SG construction}& \multicolumn{2}{|c|}{ }\\ \hline
4 &5.18&19.94 &1.90&10.30&1.89&10.1 &1.85&9.19 &1.78&9.93\\
8  &7.12& 23.58&2.65&12.85&2.61& 12.39&2.54&11.36 &2.45&12.29\\
12 &8.47&25.73 &3.15&14.44&3.12&14.95 &3.03&12.95&2.92 &13.48 \\
  16 &9.53&27.29&3.55&15.62&3.49& 15.03&3.40&13.84& 3.28& 14.98 \\
  \hline
\end{tabular}
\end{center}
\end{table}

\begin{table}
\caption{Performance of Stochastic BDDC preconditioner   for a change
  of the subdomain local problem size $\frac{H}{h}$ and $d$  with $64$ subdomains, $\sigma^2=0.5$, $l=1$, $N_{KL}=1$ . }\label{Tab:exp-2}
\begin{center}
\begin{tabular}{|c|cc||cc|cc|cc||cc|}\hline\hline
&\multicolumn{2}{|c||}{MPC}&\multicolumn{2}{|c|}{ $d=4$}&\multicolumn{2}{|c|}{ $d=5$}
                                                    &\multicolumn{2}{|c||}{ $d=6$} & \multicolumn{2}{|c|}{ Exact}\\ \hline
$\frac{H}{h}$ & Cond. & Iter. & Cond. & Iter.  & Cond. & Iter. & Cond. & Iter.   & Cond. & Iter.  \\ \hline
&\multicolumn{2}{|c||}{}&\multicolumn{6}{|c||}{ with the SC construction}& \multicolumn{2}{|c|}{ }\\ \hline
 4  &5.18&19.94 &1.95&10.83&1.87&10.42  &1.85&10.37 &1.78&9.93 \\
8   &7.12& 23.58&2.70&13.55&2.59& 13.30&2.57&13.25 &2.45&12.29\\
12 &8.47&25.73 &3.23&15.32&3.09&15.07 &3.06&14.95&2.92 &13.48 \\
16 &9.53&27.29&3.64&16.41&3.48& 15.91&3.44&15.77& 3.28& 14.98\\
 \hline
&\multicolumn{2}{|c||}{}&\multicolumn{6}{|c||}{ with the SG construction}& \multicolumn{2}{|c|}{ }\\ \hline
  4  &5.18&19.94  &1.90&10.30&1.88&10.22  &1.88&10.21 &1.78&9.93  \\
8   &7.12& 23.58&2.70&13.55&2.63& 12.72&2.63&12.70 &2.45&12.29\\
12 &8.47&25.73 &3.15&14.44&3.12&14.32&3.12&14.32&2.92 &13.48\\
16 &9.53&27.29&3.55&15.62&3.51& 15.49&3.51&15.48& 3.28& 14.98\\
  \hline
\end{tabular}
\end{center}
\end{table}


\begin{table}
\caption{Performance of Stochastic BDDC preconditioner  for a change
  of $\sigma^2$ with $64$
  subdomains, $\frac{H}{h}=8$, $l=1$, $N_{KL}=3$, $d=6$ . }\label{Tab:sig}
\begin{center}
\begin{tabular}{|c|c|c|c|c|}\hline\hline
$\sigma^2$ & MPC Iter. &  Stochastic Iter. (SC) &  Stochastic Iter. (SG)  & Exact  Iter.  \\ \hline
 0.2  &19.21 &12.07&10.08&12.05 \\
0.5  & 23.58&12.57&10.46&12.29\\
1&30.01 &14.13&12.05&12.58\\
 \hline
\end{tabular}
\end{center}
\end{table}


\begin{table}
\caption{Performance of Stochastic BDDC preconditioner  for a change
  of the number of the subdomain $N_s$ and $N_{KL}$ with 
  $\frac{H}{h}=8$, $\sigma^2=0.5$, $l=0.1$, $d=4$. }\label{Tab:exL1-1}
\begin{center}
\begin{tabular}{|c|c||c|c|c||c|}\hline\hline
&\multicolumn{1}{|c||}{MPC}&\multicolumn{1}{|c|}{ $N_{KL}=2$}
                                                    &\multicolumn{1}{c|}{ $N_{KL}=3$} & \multicolumn{1}{c||}{ $N_{KL}=4$} & \multicolumn{1}{c|}{ Exact}\\ \hline
$N_s$ & Iter. & Iter. & Iter. & Iter. & Iter.  \\ \hline
&&\multicolumn{3}{|c||}{ with the SC construction}&\\ \hline
 4  &30.84& 15.75 &12.51 &12.00&11.12 \\
8   &39.71& 16.43&14.13 &14.17&13.02\\
12  &44.63&16.51&14.89& 14.89&13.01 \\
16 &47.99& 16.42&15.63&15.63& 13.00\\\hline
&&\multicolumn{3}{|c||}{ with the SG construction}&\\ \hline
 4  &30.84& 17.47&13.67 &13.18& 11.12\\
8   &39.71& 16.37& 13.32&13.15&13.02\\
12 &44.63&15.88&13.06& 13.01&13.01 \\
16&47.99&14.87 &13.07&13.08& 13.00\\
  \hline
\end{tabular}
\end{center}
\end{table}

\begin{table}
\caption{Performance of Stochastic BDDC preconditioner  for  a change
  of the number of the subdomain $N_s$ and $d$ with 
  $\frac{H}{h}=8$, $\sigma^2=0.5$, $l=0.1$, $N_{KL}=3$. }\label{Tab:exLp-3}
\begin{center}
\begin{tabular}{|c|c||c|c|c||c|}\hline\hline
&\multicolumn{1}{|c||}{MPC}&\multicolumn{1}{|c|}{ $d=4$}
                                                   &\multicolumn{1}{c|}{
                                                     $d=5$} &
                                                              \multicolumn{1}{c||}{ $d=6$} &
                                                                                             \multicolumn{1}{c|}{ Exact}\\ \hline
$N_s$ & Iter. & Iter. & Iter. & Iter. & Iter.  \\ \hline
&&\multicolumn{3}{|c||}{ with the SC construction}&\\ \hline
 4  &30.84& 12.51 &12.43 &12.40&11.12 \\
8   &39.71& 14.13&13.30&13.09&13.02\\
12  &44.63&14.89&13.57&13.08 &13.01 \\
16 &47.99& 15.62&13.85&13.11 & 13.00\\
   \hline
&&\multicolumn{3}{|c||}{ with the SG construction}&\\ \hline
4  &30.84& 13.67 &13.61&13.61&11.12 \\
8  &39.71& 13.32&12.86&12.75&13.02\\
12 &44.63&13.06&12.38&12.07 & 13.01\\
16 &47.99& 13.07&12.13(*)&11.86 &13.00 \\\hline

\end{tabular}
\end{center}
\end{table}

\begin{table}
\caption{Performance of Stochastic BDDC preconditioner   for a change
  of the subdomain local problem size $\frac{H}{h}$ and $N_{KL}$ with  $64$
  subdomains, $\sigma^2=0.5$, $l=0.1$, $d=4$. }\label{Tab:exL1-2}
\begin{center}
\begin{tabular}{|c|c||c|c|c||c|}\hline\hline
&\multicolumn{1}{|c||}{MPC}&\multicolumn{1}{|c|}{ $N_{KL}=2$}
                                                    &\multicolumn{1}{c|}{ $N_{KL}=3$} & \multicolumn{1}{c||}{ $N_{KL}=4$} & \multicolumn{1}{c|}{ Exact}\\ \hline
$\frac{H}{h}$ & Iter. & Iter. & Iter. & Iter. & Iter.  \\ \hline
&&\multicolumn{3}{|c||}{ with the SC construction}&\\ \hline
4  &35.36& 12.73 &11.32 &11.33&10.07 \\
8   &39.71& 16.43&14.13 &14.17&13.02\\
12  &42.31&18.70&15.80& 15.85&14.09\\
16 &44.29& 20.10&17.08& 17.06& 15.71\\
  \hline
&&\multicolumn{3}{|c||}{ with the SG construction}&\\ \hline
4  &35.36& 12.50&10.48&10.43&10.07\\
8  &39.71&16.37 &13.32&13.15&13.02\\
12  &42.31&18.67&15.02&14.82&14.09\\
16 &44.29& 19.89&16.01&15.83&15.71 \\\hline
\end{tabular}
\end{center}
\end{table}

\begin{table}
\caption{Performance of Stochastic BDDC preconditioner  for a change
  of the subdomain local problem size $\frac{H}{h}$ and $d$ with $64$ subdomains, $\sigma^2=0.5$, $l=0.1$, $N_{KL}=3$ . }\label{Tab:exLp-4}
\begin{center}
\begin{tabular}{|c|c||c|c|c||c|}\hline\hline
&\multicolumn{1}{|c||}{MPC}&\multicolumn{1}{|c|}{ $d=4$}
                                                    &\multicolumn{1}{c|}{ $d=5$} & \multicolumn{1}{c||}{ $d=6$} & \multicolumn{1}{c|}{ Exact}\\ \hline
$\frac{H}{h}$ & Iter. & Iter. & Iter. & Iter. & Iter.  \\ \hline
&&\multicolumn{3}{|c||}{ with the SC construction}&\\ \hline
 4  &35.36& 11.51 &10.57 &10.30&10.07 \\
8   &39.71& 14.13&13.30&13.09&13.02\\
12  &42.31&15.80&15.16&14.96 &14.09\\
16 &44.29& 17.08&16.23& 16..01& 15.71\\
  \hline
&&\multicolumn{3}{|c|}{ with the SG construction}&\\ \hline
4& 35.36&10.48&10.11& 10.04& 10.07\\
  8   &39.71&13.32&12.86&12.75 &13.02\\
  12  &42.31&15.02&14.53& 14.47&14.09\\
16&44.29&16.01&15.59&15.50 & 15.71\\\hline
\end{tabular}
\end{center}
\end{table}
\begin{table}
\caption{Performance of Stochastic BDDC preconditioned
  inexact Schur complement for a change
  of $N_{KL}$  with $64$
  subdomains, $\frac{H}{h}=8$, $l=1$, $d=4$ . }\label{Tab:inexactFcN}
\begin{center}
\begin{tabular}{|c|cc||cc|ccc|}\hline\hline
&\multicolumn{2}{|c||}{Exact}&\multicolumn{2}{|c|}{
                               SG}&\multicolumn{3}{|c|}{ Inexact SG}
                                                  \\ \hline
$N_{KL}$& Cond. & Iter. & Cond. & Iter. & Cond. & Iter.  &  Error\\ \hline
1 &2.45&12.29&2.65&12.85&2.57&12.04&1.18e-2\\
2  &&&2.61&12.39&2.57&11.68&1.15e-2\\
3   &&&2.54&11.36&2.57&11.72&1.10e-2\\
  \hline
\end{tabular}
\end{center}
\end{table}

\begin{table}
\caption{Performance of Stochastic BDDC preconditioned
  inexact Schur complement for a change
  of $p$  with with $64$
  subdomains, $\frac{H}{h}=8$, $l=1$, $N_{KL}=1$ . }\label{Tab:inexactFcP}
\begin{center}
\begin{tabular}{|c|cc||cc|ccc|}\hline\hline
&\multicolumn{2}{|c||}{Exact}&\multicolumn{2}{|c|}{
                               SG}&\multicolumn{3}{|c|}{ Inexact SG}
                                                  \\ \hline
$p$& Cond. & Iter. & Cond. & Iter. & Cond. & Iter.  &  Error\\ \hline
4 &2.45&12.29&2.65&12.85&2.57&12.04&1.18e-2\\
5  &&&2.63&12.72&2.51&11.59&5.35e-3\\
6   &&&2.63&12.70&2.51&11.45&3.82e-3\\
  \hline
\end{tabular}
\end{center}
\end{table}

\begin{table}
\caption{Performance of Stochastic BDDC preconditioned
  inexact Schur complement for a change
  of $N_{KL}$  with $64$
  subdomains, $\frac{H}{h}=8$, $l=0.1$, $d=4$ . }\label{Tab:inexactFcNL}
\begin{center}
\begin{tabular}{|c|cc||cc|ccc|}\hline\hline
&\multicolumn{2}{|c||}{Exact}&\multicolumn{2}{|c|}{
                               SG}&\multicolumn{3}{|c|}{ Inexact SG}
                                                  \\ \hline
$N_{KL}$& Cond. & Iter. & Cond. & Iter. & Cond. & Iter.  &  Error\\ \hline
2 &2.46&13.02&3.38&16.37&2.79&14.20&1.26e-2\\
3  &&       &2.68&13.32&2.76&13.84&8.60e-3 \\
4   &&      &2.67&13.15&2.74&13.70&8.58e-3\\
  \hline
\end{tabular}
\end{center}
\end{table}

\begin{table}
\caption{Performance of Stochastic BDDC preconditioned
  inexact Schur complement for a change
  of $p$  with $64$
  subdomains, $\frac{H}{h}=8$, $l=0.1$, $N_{KL}=3$ . }\label{Tab:inexactFcPL}
\begin{center}
\begin{tabular}{|c|cc||cc|ccc|}\hline\hline
&\multicolumn{2}{|c||}{Exact}&\multicolumn{2}{|c|}{
                               SG}&\multicolumn{3}{|c|}{ Inexact SG}
                                                  \\ \hline
$p$& Cond. & Iter. & Cond. & Iter. & Cond. & Iter.  &  Error\\ \hline
4 &2.46&13.02&2.68&13.32&2.76&13.82&6.40e-3\\
5                 &&&2.56&12.86&2.69&13.09&2.98e-3\\
  6      &&         &2.55&12.75&2.53&12.72&1.04e-3\\
  \hline
\end{tabular}
\end{center}
\end{table}

\bibliography{rs}
\bibliographystyle{plain}
\end{document}